\documentclass[10pt]{amsart}
\usepackage{calrsfs}

\usepackage{amsmath,tikz}
\usetikzlibrary{tikzmark,fit}
\usetikzlibrary{arrows,calc}
\usetikzlibrary{decorations.pathreplacing,decorations.markings}
\usepackage{hyperref}

\usepackage{mathdots}
 \usepackage[greek,english]{babel}
\usepackage{graphicx}
\usepackage{amsthm}
\usepackage{amssymb}
\usepackage{multirow}
\usepackage{tikz-cd}
\usepackage{amsmath}
\usepackage{todonotes}
\usepackage{amsbsy}
\usepackage[all]{xy}
\usepackage{qtree}

\usepackage{enumitem}
\usepackage{xfrac}    
\usepackage{color, colortbl}
\definecolor{LightCyan}{rgb}{0.88,1,1}
\definecolor{Gray}{gray}{0.9}

\usepackage{faktor}

\usepackage{changes}
\definechangesauthor[color=orange,name={Aristidesd Kontogeorgis}]{AK}
\definechangesauthor[color=blue,name={Alex Terezakis}]{AT}

\usepackage{longtable,booktabs}

\usepackage{float} 

\definecolor{color_29791}{rgb}{0,0,0}
\newcommand{\dotsize}{0.5} 

\newtheorem{theorem}{Theorem}
\newtheorem{lemma}[theorem]{Lemma}

\newtheorem{proposition}[theorem]{Proposition}
\theoremstyle{definition}
\newtheorem{example}[theorem]{Example}
\newtheorem{remark}[theorem]{Remark}
\newtheorem{definition}[theorem]{Definition}

\newtheorem{method}[theorem]{Method}

\newcommand{\Ker}{\mathrm Ker}

\newcommand{\Z}{\mathbb{Z}}
\newcommand{\R}{\mathbb{R}}

\newcommand{\im}{{ \mathrm Im }}

\renewcommand{\mod}{{\;\mathrm{mod}}}

\newcommand{\pii}{q}

\date{\today}

\title{On cyclic group covers of the projective line}

\author[G. Katsimprakis]{George Katsimprakis}
\address{Department of Mathematics, National and Kapodistrian University of Athens\\
Panepistimioupolis, 15784 Athens, Greece}
\email{katsimprakis@yahoo.gr}

\author[A. Kontogeorgis]{Aristides Kontogeorgis}
 \address{Department of Mathematics, National and Kapodistrian  University of Athens
 Panepistimioupolis, 15784 Athens, Greece}
 \email{kontogar@math.uoa.gr}

\date \today

\makeatletter
\newcommand{\aprod}{\mathop{\operator@font \hbox{\Large$\ast$}}}
\makeatother

\begin{document}



\begin{abstract}

This article extends the study of cyclic ramified covers of the projective line defined by Kummer equations. We consider the most general case of such covers, allowing arbitrary orders  in the roots of the generating radicant. The primary goal is the computation of the fundamental group of both the open and complete curve. We employ tools of combinatorial group theory  utilizing the Smith Normal Form. This result is further visualized through the theory of foldings  and $S$-graphs. Finally, we apply the theory of Alexander modules  and the Crowell exact sequence to compute the abelianization of the fundamental group, $H_{1}(X, \mathbb{Z})$, and determine its Galois~module~structure over a field $k$  confirming the result using the Chevalley-Weil formula.
\end{abstract}
\maketitle

\tikzset{middlearrow/.style={
        decoration={markings,
            mark= at position 0.5 with {\arrow{#1}} ,
        },
        postaction={decorate}
    }
}


%
%
%
\section{Introduction}

It is known that information about an algebraic curve and especially information about the actions of the automorphism groups, the mapping class group, and the absolute Galois group on the homology of the curve can by studied by determining the fundamental group of an open covering of a curve, 
\cite{MR4117575}, \cite{MR4186523}, \cite{kontogeorgis2024galoisactionhomologyheisenberg}.

In \cite{MR4117575} the second author and P. Paramantzoglou considered the actions  defined as Kummer covers of the projective line given by the equation 
\[
y^n= \prod_{i=1}^s (x-b_{i}). 
\]
In that setting we have a cyclic ramified cover of the projective line, ramified fully above $s$-points.  An essential part of that article was the computation of the fundamental group both of the corresponding topological cover and of the complete curve. In this article we will extend our study to the most general of cyclic covers of the projective line, by allowing arbitrary orders in the roots of the right hand side of the above equation.

The Kummer equation defines the curve as a $\mathbb{Z}/n\mathbb{Z}$-Galois cover of the projective line $\mathbb{P}^1$.
Riemann's Existence Theorem provides a crucial link between algebraic geometry and topology, particularly in the study of algebraic curves and their coverings. The theorem essentially asserts that every finite, connected, topological covering space of a compact Riemann surface (with a finite number of punctures allowed) corresponds to an algebraic function field extension of the function field of the base curve. For the cyclic ramified covers of the projective line studied here, the theorem is implicitly at work, establishing that the algebraic Kummer cover $X^0 \to Y_0$ (where $Y_0 = \mathbb{P}^1 \setminus S$ is the punctured sphere) is equivalent to a topological Galois cover. This correspondence allows us to translate the geometric problem of the cover's structure into a group-theoretic problem involving the monodromy action of the fundamental group of the base space, $\pi_1(Y_0, y_0)$, on the cover's fibers.

The successful determination of $\pi_1(X^0, x_0)$ is not just a group-theoretic result; it is the necessary and foundational step for studying the algebraic properties of the corresponding complete curve. Specifically, this result allows us to obtain crucial information on the Galois module structure of the first homology group, $H_1(X, \mathbb{Z})$, by analyzing the abelianization of $\pi_1(X^0, x_0)$ and its quotient relative to the branch point relations. The later sections of this article leverage this result to compute the abelianization and the Galois module structure of the homology group.

\bigskip 
\noindent 
{\bf Notation.}
Set $\overline{d}=(d_1, \ldots  d_{s-1})$ and consider the unique smooth projective curve $X_{n,\bar{d}}$ defined over complex numbers, corresponding to the function field given by the Kummer equation
\begin{equation}
\label{eq:def-curve-n}
	y^n=\prod_{i=1}^s (x-b_i)^{d_i}, \qquad (d_i,n)\neq n.
\end{equation}
The ramification points are the roots $x=b_i$, which are ramified with ramification index $e_i=\frac{n}{(n,d_i)}$. Thus, if for a given $1\leq i \leq s$ we have $(n,d_i)=1$, then the point $x=b_i$ is fully ramified, while the condition  $(n,d_i)\neq n$, for all $1\leq i \leq n$ ensures that all points $x=b_i$ are ramified.

Without loss of generality, we can assume that the point at infinity is not ramified, this is equivalent to the condition 
\begin{equation} \label{eq:degree0}
    \sum_{i=1}^s d_i \equiv 0 \pmod n,
\end{equation} see \cite[p. 667]{Ko:99}.

We can also assume that the greatest common divisor $d=(d_1, \ldots , d_s)$ is prime to $n$. Otherwise, the curve equation  can be written as 
\[
    y^n = 
    \left(
        \prod_{i=1}^s(x-b_i)^{\frac{d_i}{d}}
        \right)^d
\]
and   the curve $X_{n,\bar{d}}$ is a union of curves determined by by the equations 
\[ 
y^{\frac{n}{\delta }} = \zeta_\delta^\nu \left( 
    \prod_{i=1}^s(x-b_i)^{\frac{d_i}{d}}
    \right)^{\frac{d}{\delta }},
\]
where $\delta=(n,d)$ and  $\zeta_\delta$ is a primitive $\delta$-root of unity and $0\leq \nu < \delta$. 
We thus see that if  $\delta>1$, then the original curve is not irreducible.

Denote for simplicity of notation $X_{n,\bar{d}}$ by $X$. 
The curve $X$ can be realized as a ramified cover 
$\psi:X \rightarrow  \mathbb{P}^1$ of the projective line, with branch locus $S=\{P_{x=b_1}, \ldots , P_{x=b_s}\}$. Set $X^0=X \setminus \psi^{-1}(S)$, and 
$Y_0=\mathbb{P}^1 \setminus S$. For an arbitrary point $y_0 \in Y_0$ it is known that $\pi_1(Y_0,y_0)=F_{s-1}$, where 
\[
F_{s-1}= \langle x_1, \ldots ,x_s | x_1 x_2 \cdots x_{s-1}=1\rangle
\]
is a free group generated by the loops $x_1, \ldots ,x_{s}$ starting from the point $y_0$, each one circling around each point of $S$. The elements $x_1, \ldots ,x_{s-1}$ are free generators of $F_{s-1}$ since $x_s= x_{s-1}^{-1} \cdots x_1^{-1}$. From now on, by abuse of notation, we will consider $F_{s-1}=\langle x_1, \ldots ,x_{s-1} \rangle$. 

The open cover $X^0 \rightarrow \mathbb{P}^1 -S$ is a topological Galois cover with Galois group $C_n= \pi_1(Y_0,y_0)/N$, 
for a normal subgroup $N=\pi_1(X^0,x_0)$, which we are going to compute.

\begin{theorem}
    \label{th:funddef}
Set  \( d=  (d_1, \ldots, d_{s-1}) \) and suppose that $(d,n)=1$. 
Consider the natural epimorphism  
    \[
    \pi: d\mathbb{Z} \rightarrow \frac{ \mathbb{Z}}{n \mathbb{Z}}
    \]
    and the map  
\begin{align*}
    \alpha_{\overline{d} }: F_{s-1} &
    \longrightarrow d\mathbb{Z}
\\
x_i & \longmapsto d_i  
\end{align*}
where $\overline{d}=(d_1, \ldots ,d_{s-1})$.  
The fundamental group  $\pi_1(X^0,x_0)= \ker \pi\circ\alpha_{d_1, \ldots ,d_{s-1}}$.
\end{theorem}

\begin{remark}
    For $\overline{1}=(1, \ldots ,1)$, the map
$\alpha_{\overline{1} }$ is the winding map, see also \cite[sec. 4]{MR4117575}.
\end{remark}
\begin{remark}
In the definition of $\alpha_{\overline{d} }$ we have used only the information of the exponents $d_1, \ldots , d_{s-1}$ and not the information of the exponent $d_s$, which also plays a role in the ramification of the point $P_{x=b_s}$. For the loop $x_s$ surrounding the point $P_{x=b_s}$ we have 
$x_{s} = x_{s-1}^{-1} x_{s-2}^{-1} \cdots x_{2}^{-1} x_1^{-1}$. When we consider the map 
$\pi \circ \alpha_{\overline{d} }$ the condition (\ref{eq:degree0}) implies that 
\[
    d_s=\alpha_{\overline{d}} (x_s) =
    - \sum_{\nu=1}^{s-1} d_\nu =
    - \sum_{\nu=1}^{s-1} \alpha_{\overline{d}}  (x_\nu) \pmod n
\]
\end{remark}
In \cite{MR4117575} the groups 
\begin{align*}
R_{n,s-1} &=\ker( \pi \circ \alpha_{\overline{1} }) \\
R_{0,s-1} &=\ker(\alpha_{\overline{1} })
\end{align*} 
are studied using Schreier's lemma and it is proved that 
\begin{align*}
R_{n,s-1} &= \langle  \{ x_1^{i} x_j x_1^{-i-1}: 0 \leq i \leq n-2, 2 \leq j \leq s-1 \} \cup \{x_1^{n-1}x_j:1 \leq j  \leq s-1\} \rangle \\
R_{0,s-1} &= \langle x_1^i x_j x_1^{-i-1}: i \in \mathbb{Z}, j=2, \ldots ,s-1\rangle.
\end{align*}

Applying Schreier lemma in the more general case is a difficult task and we will use two methods in order to make progress in this problem. Essentially the computation of the fundamental group reduces to solving a linear Diophantine equation, which will be solved in proposition \ref{prop:paramDioph}, using Smith normal form. 
Following the parametrization of solutions of the Diophantine equation we give a new set of generators $y_{1},\ldots,y_{s-1}$ of the free group $F_{s-1}$ and a transversal set $T$, 
that is a set of reduced words such that each right coset of $N$ in $F_{s-1}$ contains a a unique word of $T$ and all initial segments of these words also lie in $T$, see \cite[def. 8.9]{bogoGrp}. By  applying Schreier's lemma we arrive at the following   
\begin{theorem}
    \label{th:Sch}
    Let $y_1, \ldots , y_{s-1}$ be the generators of the free group $F_{s-1}$ given by eq. (\ref{eq:ygens}). 

    \noindent$\bullet$ 
    A set of generators for the free group $\mathrm{\ker}\pi \circ \alpha_{\overline{d} }$ is given by 
    \[
     \{ y_1^{\nu} y_j y_1^{-\nu}: 0\leq \nu < n, 2\leq j \leq s-1\} \cup 
     \{y_1^n\}. 
     \]
    The group $\mathrm{\ker } \pi \alpha_{\overline{d} }$ is a free group of rank $(s-2)n+1$. 

    \noindent$\bullet$
    A set of generators for the group $\mathrm{\ker} \alpha_{\overline{d} }$ is given by 
    \[
    \{ y_1^{\nu} y_j y_1^{-\nu}: \nu \in \mathbb{Z}, 2\leq j 
    \leq s-1\}.    
    \]   
\end{theorem}

Our second approach to this problem involves 
the theory of foldings in order to study $\ker \alpha _{\overline{d} }$ as an intersection of two known groups namely the group $\mathrm{ker}{\pi \circ \alpha_{\bar{1}}}$ (resp. $\mathrm{ker}{\alpha_{\bar{1}}}$)
and $\langle x_{1}^{d_{1}}, \ldots, x_{s-1}^{d_{s-1}} \rangle$. Although, this method leads eventually to the same Diophantine equations we have included it as well since it provides us with a better geometric visualization of the fundamental group in question.   

The structure of the article is as follows. 
In section \ref{sec:monodromy} we relate the functions $\alpha_{\bar{d}}$ and $\pi \circ \alpha_{\bar{d}}$ to the ramification of the cover $X_{n,\bar{d}} \rightarrow  \mathbb{P}^1$. The fundamental groups of the open curve $X^0$ is related to the computation of the kernel of $\mathrm{ker}(\alpha_{\bar{d}})$. In section \ref{sec:smith} we employ the theory of Smith normal form in order to solve a system of linear Diophantine equations corresponding to the computation of the above kernel in an abelianized setting. 
In section \ref{sec:schreier} we use the information of the Smith normal form in order to construct a Schreier transversal set and eventually a set of generators of the desired fundamental group. In section \ref{sec:folding} we use the theory of folding in order to arrive to the kernels $\alpha_{\bar{d}}$ and $\pi \circ \alpha_{\bar{d}}$ by representing them as intersection of the fundamental group of the curve $X_{n,\bar{1}}$ and the group 
$x_1^{d_1}, ,\ldots, x_{s-1}^{d_{s-1}}$. In section \ref{sec:braid} we study whether the braid group realized as the mapping class group of $\mathbb{P}^1 \backslash \{b_1, \ldots, b_{s}\}$ can be lifted to the curve $X_{n,\bar{d}}$ and we give a necessary and sufficient condition for the lift. 

Finally in section \ref{sec:compact} we construct the fundamental group of the complete curve  and using the theory of Alexander modules \cite{Morishita2011-yw} we compute its abelianization and the Galois module structure of the homology group. In \cite{MR4223071}
the theory of Alexander modules (or $\Psi$-differential modules) is reinterpreted  within the framework of non-commutative differential modules. This work was directly motivated by geometric problems, specifically the study of Galois coverings of curves, see also \cite{MR4117575}, \cite{MR4186523}, \cite{kontogeorgis2024galoisactionhomologyheisenberg}. For the Kummer cover article, the Alexander module $\mathcal{A}_{\psi}$ is the essential tool used to understand the homology group $H_1(X, \mathbb{Z})$ as a $\mathbb{Z}[C]$-module. The work done in \cite{MR4223071}  provides the rigorous algebraic foundation for this application by proving that the non-commutative module of differentials, which represents derivations, coincides with the Alexander module. 

The final section of the article connects the group-theoretic computation of the fundamental group to the Galois module structure of the homology group $H_1(X, \mathbb{Z})$ by analyzing the $k[C]$-module structure of $H_1(X, k)$ over a field $k$ with characteristic $p$ where $(p, n)=1$. This analysis culminates in proposition \ref{prop:35}, which determines the multiplicity $M_\nu$ of each irreducible character $\chi_\nu$ in the decomposition of the homology group $H_1(X, k)$. Crucially, this result is confirmed by comparing it with the Chevalley-Weil formula (used for the dual space of regular differentials, $H^0(X, \Omega_X)$). This comparison is justified by the Hodge Decomposition and Serre~Duality theorems. Specifically, the total multiplicity $M_\nu$ in $H^1(X, \mathbb{C})$ (which is dual to $H_1(X, \mathbb{C})$) is shown to be the sum of the multiplicities of the characters in holomorphic  and the anti-holomorphic forms. This consistency between the combinatorial group theory approach (via Alexander modules) and the analytical approach (via Chevalley-Weil) validates the final formula for the homology module structure.

\noindent
\textbf{Acknowledgements:}
The authors are grateful to the referee for his careful reading of the manuscript and for the valuable remarks and corrections that substantially improved the clarity and correctness of the paper.

\section{Monodromy actions}
\label{sec:monodromy}

We will now prove theorem \ref{th:funddef}.
Fix the point $P=P_{x=b_i}$  of $\mathbb{P}^1$ and fix a point $P_{\nu}$ in the set of points
 $\{P_1, \ldots , P_{(n,d_i)}\}$ above $P$.   Let  $t_\nu$ be a local uniformizer at $P_\nu$, and let $\mathbb{C}[[t_\nu]]$ be the completed local ring at $P$, which does not depend on the selection of the local uniformizer $t_\nu$. Since $P_\nu/P$ is ramified with ramification index $e_i=\frac{n}{(n,d_i)}$, we might assume that $x-b_i=t_\nu^{e_i}$ in the ring $\mathbb{C}[[t_\nu]]$. Indeed, the valuation $v_{P_\nu}(x-b_i)=e_i$ and by Hensel's lemma, every unit is an $n$-th power that can be absorbed by reselecting the uniformizer $t_\nu$ if necessary. We replace the factor $(x-b_i)^{d_i} = t_\nu^{e_i d_i}$ in the defining equation (\ref{eq:def-curve-n}) in order to arrive at the equation
\begin{equation}
\label{eq:def-eq-rep-t}
y^n=t_\nu^{e_id_i}U_i, \qquad U_i=\prod_{\mu=1 
\atop \mu \neq i}^s (x-b_\mu)^{d_\mu}\in \mathbb{C}[x], v_{P_\nu}(U_i)=0.
\end{equation} 
The Galois group of the extension $\mathbb{C}(X)/\mathbb{C}(x)$ 
 is cyclic, and the cyclic group 
is generated by the element $\sigma$ such that $\sigma(y)=\zeta_n y$, for some fixed primitive root of unity $\zeta_n$. Since $U_i \in \mathbb{C}[x]$ we have that $\sigma(U_i)=U_i$. Let $u_i\in \mathbb{C} [[t_\nu]]$ be an  $n$-th root of $U_i$. 
Unless $(n,d_i)=1$,  there is no well defined action of $\sigma$ on $t_\nu$, since $\sigma$ permutes the points extending $P$. On the other hand there is a well defined action of $\sigma^{(n,d_i)}$ on $k[[t_\nu]]$.
We will prove that 
$\sigma^{(n,d_i)}(u_i)=u_i$. 
Indeed,  $\sigma(u_i)^n= \sigma(U_i)= u_i^n$, so $\sigma(u_i)=\zeta_n^{\xi} u_i$, for some exponent $0 \leq \xi <n$. 
Since $u_i$ is a unit in $\mathbb{C}[[t_\nu]]$ it is of the form $u_i=a_0^{(i)}+ a_1^{(i)} t_\nu+\cdots$, with $a_0^{(i)}\neq 0$, that is $u_i\equiv a_0^{(i)} \pmod{t_{\nu} k[[t_\nu]]}$. 
Observe that $\sigma(a_0^{(i)})=a_0^{(i)}$, and
 $
 \sigma^{(n,d_i)} $
  induces an action on $\mathbb{C}[[t_\nu]]/ t_\nu \mathbb{C}[[t_\nu]]$, which reduces to the trivial action of $ \sigma^{(n,d_i)} $ on $\mathbb{C}$, so 
we finally obtain that $ (n,d_i)\xi\equiv 0 \pmod n$, i.e. $u_i$ is $\sigma^{(n,d_i)}$-invariant.

Select the primitive $e_i$ root of unity  $\zeta_{e_i}$ by  $\zeta_{e_i}=\zeta_n^{(n,d_i)}$.
The action of  
 $\sigma^{(n,d_i)}$ on 
$t_\nu$ is given by 
$\sigma^{(n,d_i)}(t_\nu)=\zeta_{e_i}^{\ell_{i,\nu} }t_\nu= \zeta_n^{\ell_{i,\nu}(n,d_i)}t_\nu$ for some $\ell_{i,\nu} \in \mathbb{N}$.
We will now compute $\ell_{i,\nu }$. By considering the $n-th$ root of eq. (\ref{eq:def-eq-rep-t}) we have that 
\[
	y = t_\nu^{\frac{d_i}{(n,d_i)}} u_i. 
\]
In the above equation we have absorbed the $n$-th root of unity that appears after taking the $n$-th root into the unit $u_i$. 
Since 
 by assumption
$\sigma(y)=\zeta_n y$ we have
\[
    \zeta_n^{(n,d_i)}y= \sigma^{(n,d_i)}y=\sigma^{(n,d_i)} t_\nu^{\frac{d_i}{(n,d_i)}} u_i 
    =
    \zeta_n^{\ell_{i,\nu}(n,d_i) \frac{d_i}{(n,d_i)} } t_\nu^{\frac{d_i}{(n,d_i)}} u_i  =
    \zeta_n^{\ell_{i,\nu} d_i } y.
\]
We  thus have  
\begin{equation}
    \label{eq:ldrel}
    \ell_{i,\nu} d_i \equiv (n,d_i) \pmod n 
    \Rightarrow \ell_{i,\nu} \frac{d_i}{(n,d_i)} \equiv 1 \mod \frac{n}{(n,d_i)}.
\end{equation}
Since, $\left( \frac{d_i}{(n,d_i)},\frac{n}{(n,d_i)}\right)=1$, the above equation has unique solution
\begin{equation}
\label{eq:elrel}
    \ell_{i,\nu} \equiv  \left( \frac{d_i}{(n,d_i)}\right)^{-1} \pmod{\frac{n}{(n,d_i)}} , 
\end{equation}
and does not depend on $\nu$. So we will simplify the notation by setting $\ell_i=\ell_{i,\nu }$.

\begin{remark}
\label{remark-action-points-functions}
Consider a group $G$ acting on a curve $X$. This action defines an action on functions $f:X \rightarrow \mathbb{C}$, that is on the function field $\mathbb{C}(X)$ of the curve $X$ as follows: The function $f$ is mapped to the function $f \circ \sigma^{-1}$. This is natural since the point $P$ can be characterized by the maximal ideal in an affine neighborhood of the point of functions vanishing at $P$. Therefore, if $f$ vanishes at $P$ then $f\circ \sigma^{-1}$ vanishes at $\sigma(P)$. 
By abuse of notation we will use both $\sigma(P)$, when $P\in X $ and $\sigma(f)$, when $f\in \mathbb{C}(X)$, where $\sigma(f)(P)=f(\sigma^{-1}P)$.
\end{remark}
The open curve $X^0$ is a topological cover of $Y_0$, hence it is acted on 
by the group $\pi_1(Y_0,y_0)$ in terms of the monodromy action.
As before fix the point $P_i=P_{x=b_i}$ for some $1 \leq i \leq s-1$ and consider the set of points 
$P_1^{(i)}, \ldots , P_{(n,d_i)}^{(i)}$ above $P_i$. 
There is an open neighborhood $V_0$ of $P_i$  and open neighborhoods $V_\nu$ of the points $P^{(i)}_\nu$, $1 \leq \nu \leq (n,d_i)$ 
and selection of uniformizers $t_\nu$ so that 
 $t_\nu: V_\nu \rightarrow D =\{z \in \mathbb{C}: |z|<1\}$ are isomorphisms and 
 $\psi|_{V_\nu}:V_\nu \rightarrow V_0$ is given by $t_\nu \mapsto t_\nu^{e_\nu}$.    
We thus have the following diagram 
\begin{equation*}
    \xymatrix{
        V_\nu  \ar[r]^{t_\nu} \ar[d]_{\psi\mid_{V_\nu}}  & D \ar[d]^{z \mapsto z^{e_\nu}}
        \\
        V_0 \ar[r]^-{\cong} & D 
    }
\end{equation*}
In this setting the generator $x_i$ can be considered as a loop
$x_i(\tau)=r\cdot e^{2 \pi i \tau}$, $\tau\in [0,2 \pi]$ for some $r$, $\mathbb{R} \ni r<1$, so that the loop $x_i(\tau)$ is 
inside the neighborhood $D$, starting from the point $ V_0 \ni x_0=r   \in \mathbb{C}$. 
Fix points $y_1 \in V_1, \ldots ,y_{(n,d_i)} \in V_{(n,d_i)}$. 
The closed paths $x_i^\mu$ for $\mu \in \mathbb{Z}$ can be lifted to paths starting from $y_1$ and ending to points in $\psi^{-1}(x_0)$. The end point of path $x_0^\mu$ is by definition the monodromy action of $x_0^\mu$ on $y_1$. 

In our case the monodromy action can be made explicit as follows: By the inverse map theorem we can write the quantity $U_i$ defined in equation (\ref{eq:def-eq-rep-t}) as 
$U_i=v_i^{d_i}$ in a small neighborhood of the point $b_i$ so that $(x-b_i)^{d_i} U_i=
 \big( (x-b_i) v_i(x)\big)^{d_i}=z^{e_i d_i}$ and $X(x)=(x-b_i)v_i(x)$.
The defining equation of the curve can be now written as 
\[
    y^n = z^{e_i d_i}=X(x)^{d_i}. 
\] 
The above equation can be factored as 
\[
    \prod_{k=0}^{(n,d_i)-1}\left(y^{\frac{n}{(n,d_i})} - \zeta_{(n,d_i)}^k z^{e_i \frac{d_i}{(n,d_i)}} \right) =
    \prod_{k=0}^{(n,d_i)-1}\left(y^{\frac{n}{(n,d_i})} - \zeta_{(n,d_i)}^k X(x)^{ \frac{d_i}{(n,d_i)}} \right) =0.
\]
Each factor gives rise to a ramified point $P_\nu$ above the point $b_i$. Also a closed loop 
\[
  X(\tau)= \rho e^{2 \pi i \tau}, \; 0 \leq \tau \leq 1
\]
lifts to a loop 
\[
\gamma(\tau)=
\big(Y(\tau),X(\tau)\big)= 
\big( \zeta_{n}^k \rho^{1/n} e^{2 \pi i \frac{d_i}{n} \tau}, \rho e^{2 \pi i \tau} \big), \; 0 \leq \tau \leq 1
\]
starting at the point $(\zeta_{n}^k \rho^{1/e_i},\rho)$ and ending at the 
point $(\zeta_{n}^k \rho^{1/e_i} \zeta_{(n,d_i)}^{\frac{d_i}{n}},\rho)$. Here we have assumed that $\zeta_{(n,d_i)}=\zeta_{n}^{\frac{n}{(n,d_i)}}$ and when taking the $n/(n,d_i)$-root we made a choice for the starting point. The monodromy action is given by multiplying the $Y$ coordinate by $\zeta_n^{d_i}=
\zeta_{\frac{n}{(n,d_i)}}^{\frac{d_i}{(n,d_i)}}$ and multiplication by $\zeta_{\frac{n}{(n,d_i)}}$ is the same as applying $\sigma^{(n,d_i)}$. 

We have thus proved the following
\begin{lemma}
    \label{lemma:mono}
The monodromy action on points near the ramified point $b_i$ is given by $\sigma^{(n,d_i) \frac{d_i}{(n,d_i)}} =\sigma^{d_i}$.
\end{lemma}


By covering space theory, there is a group homomorphism  
$\alpha_d: \pi_1(Y_0,x_0)  \rightarrow \Z/ n\Z$. We will prove that this map can be naturally factored through a map $\alpha_{\bar{d}}: \pi_1(Y_0,x_0) \rightarrow d \mathbb{Z}$. 
as follows:

Consider the following map coming from equation (\ref{eq:defalpha})
\[
	\pi_1(Y_0,x_0) \stackrel{\alpha_{\overline{d} }}{\longrightarrow} \mathbb{Z} 
	\stackrel{\pi}{\longrightarrow}   \mathbb{Z}/n \mathbb{Z} \rightarrow 0.
\]
The information of  such a  map $\alpha=\alpha_{\bar{d}}$ can be encoded in the integers
 $a_i = \alpha(x_i)$,  which are mapped by $\pi$  to  elements in 
 $\mathbb{Z}/n  \mathbb{Z} \cong \mathrm{Gal}(X/\mathbb{P}^1)$. 
 The element $\pi(a_i) \in \mathbb{Z}/n \mathbb{Z}$  
  has order $o_i:=\frac{n}{(n,a_i)}$. 
 It is known that the image  $\pi\circ \alpha \left(\pi_1(Y_0,x_0)\right)$ acts 
transitively on the fiber $\psi^{-1}(x_0)$
by monodromy representation.  
This monodromy representation has been computed in lemma \ref{lemma:mono}
and gives as that all $a_i=d_i$. 


\begin{remark}
In \cite{MR4117575} we have studied the case $d_i=1$. 
In this case, since $(n,a_i) \mid (n,d_i)$ we have that $a_i \equiv 1 \pmod n$
and we have considered the case 
$\alpha(x_i)=a_i=1$, that is $\alpha$ is the ordinary winding number function. 

In this article, we generalize to the case where $\alpha(x_i)=d_i$, and we have  also assumed that $d=(d_1, \ldots ,d_{s-1})$ is prime to $n$. This assumption ensures as that the map $\pi \circ \alpha$ is onto $\mathbb{Z}/n \mathbb{Z}$. Indeed,  
we write $d= \mu_1 d_1 + \cdots + \mu_{s-1} d_{s-1}$, for some 
$\mu_1, \mu_2, \ldots, \mu_{d-1} \in \mathbb{Z}$ and then 
\[
\pi \circ \alpha 
\left(
    x_1^{\mu_1}\cdots x_{s-1}^{\mu_{s-1}} 
\right)
= d \pmod n
\]
Since $(d,n)=1$ we have that the order of $d$ in $\mathbb{Z}/n\mathbb{Z}$ is $n$. 
\end{remark}

%
\section{Smith normal form}
\label{sec:smith}
The problem of computing the groups $\mathrm{ker}\pi\circ\alpha_{\bar{d}}$ and $\mathrm{ker}\alpha_{d}$ is reduced to the problem of finding solutions of the linear Diophantine equations 
\begin{equation}
    \label{eq:sol1}
l_1 d_1 + \cdots + l_{s-1} d_{s-1} \equiv 0 \pmod n. 
\end{equation}
and 
\begin{equation}
    \label{eq:sol2}
l_1 d_1 + \cdots + l_{s-1} d_{s-1}  =0. 
\end{equation}

In order to solve the equations (\ref{eq:sol1}) and (\ref{eq:sol2}) we will employ the Smith normal form:
\begin{theorem}
    Given a $m \times n$  matrix  $A$  with integer entries there are invertible matrices $L \in \mathrm{SL}_m(\mathbb{Z}) $ and $R \in \mathrm{SL}_{n}(\mathbb{Z})$ so that 
    \[
        LAR = 
        \begin{pmatrix}
            D   & 0  \\
            0 & 0  
         \end{pmatrix}, 
    \] 
where $D=\mathrm{diag }(\delta _1, \ldots, \delta _r)$, with $r\leq \min(n,m)$ and $\delta _1 \mid \delta _2 \mid \cdots \mid  \delta_j, \delta_{j+1}= \cdots =\delta _r=0$.
\end{theorem}
\begin{proof}
    See \cite[th.3.8, p.181]{MR780184}.
\end{proof}

The above theorem applied to the $1 \times (s-1)$ matrix $A=(d_1, \ldots, d_{s-1})$ gives us  a matrix $R \in \mathrm{SL}_{s-1}(\mathbb{Z})$, $L\in \{-1,1\}$, so that 
\begin{equation}
    \label{eq:dSN}
    (d_1, \ldots , d_{s-1})R = (d, 0, \ldots , 0).
\end{equation}  
The integer $d$ from the Smith normal form above is the greatest common divisor of $(d_1,\ldots,d_{s-1})$ since the image of the map 
\begin{align*}
    \mathbb{Z}^{s-1} & \longrightarrow \mathbb{Z}
    \\
    (l_1, \ldots , l_{s-1} ) & \longmapsto \sum_{\nu=1}^{s-1} l_\nu d_\nu
\end{align*}
is $d \mathbb{Z}$.
\begin{definition}
We say that a subgroup \(H < F_{s-1}\) is \emph{normally generated} by the elements 
\(w_1,\ldots,w_s\) if, in addition to words in the generators \(w_1,\ldots,w_s\),  
we also include all conjugates \(x^{i} w_j x^{-i}\) for every \(i \in \mathbb{Z}\) and 
every \(x \in F_{s-1}\). The elements $w_1,\ldots,w_s$ will be called {\em normal generators}. 
\end{definition}

\begin{proposition}
    Let $R=(r_{ij})$ be the matrix of the Smith normal form for the set of integers $(d_1, \ldots , d_{s-1})$ 
    defined by eq. (\ref{eq:dSN}). A set of normal generators for the groups 
    $\mathrm{\ker }\alpha_{(d_1, \ldots ,d_{s-1})}$  and $\mathrm{\ker }\pi \alpha_{(d_1, \ldots ,d_{s-1})}$ is given by 
 \begin{align*}
    & [x_i,x_j]=x_ix_jx_i^{-1}x_j^{-1} \text{ for } 1\leq i< j \leq s-1 \\
    &x_1^{ N r_{11}} x_2^{N r_{21}} \cdots x_{s-1}^{N r_{s-1,1}}
    \\
    &x_1^{r_{12}} x_2^{ r_{22}} \cdots x_{s-1}^{ r_{s-1,2}}
    \\
    &\cdots
    \\
    &x_{1}^{r_{1,s-1}} x_2^{ r_{2,s-1}} \cdots x_{s-1}^{ r_{s-1,s-1}}
 \end{align*}
where $N=0$ in the case of eq. (\ref{eq:sol1}) and $N=n$ in the case of eq. (\ref{eq:sol2}).  
\end{proposition}
\begin{proof}
 Obsereve that the quotients $F_{s-1}/\mathrm{ker}\alpha_{(d_1,\ldots,d_{s-1})}$ and $F_{s-1}/\mathrm{ker}\pi \alpha_{(d_1,\ldots,d_{s-1})}$ therefore all commutators have to be included in the kernels. 
 The equality
\[
    A (l_1, \ldots , l_{s-1})^t = n \kappa 
\]
is equivalent to the equality
\begin{equation}
    \label{eq:easySm}
    (d, 0, \ldots , 0) (l^{\prime} _1,\ldots , l^{\prime}_{s-1})^t = n \kappa, 
\end{equation}
where $(l_1, \ldots , l_{s-1})^t 
= R (l^{\prime} _1,\ldots , l^{\prime}_{s-1})^t$. Equation (\ref{eq:easySm}) determines that $ d l^{\prime} _1 = n \kappa$ and since we have assumed that $(d,n)=1$ we have that $d\mid \kappa$,  $l^{\prime}_1 = n \frac{\kappa}{d} = n T$, for some $T \in \mathbb{Z}$. For the integers $l^{\prime} _2, \ldots , l^{\prime} _{s-1}$ eq. (\ref{eq:easySm}) does not pose any condition. 
\end{proof}
We thus arrive at the following parametrization of the  solutions of eq. (\ref{eq:sol1}) and (\ref{eq:sol2}).
\begin{proposition}
    \label{prop:paramDioph}
The solutions of eq. (\ref{eq:sol1}) are given by 
\[
    (l_1, \ldots , l_{s-1})^t = R (n t_1, t_2, \ldots , t_{s-1})^t,   \text{ where } t_1, \ldots , t_{s-1} \in \mathbb{Z}. 
\]
The solutions of eq. (\ref{eq:sol2}) are given by
\[
    (l_1, \ldots , l_{s-1})^t = R (0, t_2, \ldots , t_{s-1})^t,   \text{ where } t_2, \ldots , t_{s-1} \in \mathbb{Z}. 
\]
\end{proposition}
\begin{proof}

Denote by $r_{ij}$ the entries of  $R=(r_{ij})$. We have that 
\[
    x_{1}^{r_{11} nt_1+\sum_{\nu=2}^{s-1}  r_{1\nu} }
    x_{2}^{r_{21} nt_1+\sum_{\nu=2}^{s-1}  r_{2\nu} }
     \cdots 
    x_{s-1}^{r_{s-1,1} nt_1+\sum_{\nu=2}^{s-1}  r_{s-1,\nu} }
\]
are words of $\ker \alpha_{(d_1, \ldots ,d_{s-1})}$ while
 for $(t_1, \ldots ,t_{s-1})$ running over the rows $e_1=(1,0,\ldots ,0), e_2=(0,1,0,\ldots ,0),\ldots ,e_n=(0,\ldots ,0,1)$ of the identity matrix $\mathbb{I}_{s-1}$ we can obtain a set of generators for $\ker \pi \alpha_{(d_1, \ldots , d_{s-1})}$. The case 
 $\ker  \alpha_{(d_1, \ldots , d_{s-1})}$ is similar.
\end{proof}
\begin{definition}
    \label{def:H}
    For a tuple $\overline{d}=(d_1, \ldots , d_{s-1})$ we will denote by 
    \begin{align*}
        H^{\overline{d}} &= \frac{ \ker  \alpha_{(d_1, \ldots , d_{s-1})}  }{ F_{s-1}^{\prime} }
        \\
        H_n^{\overline{d}} &= \frac{ \ker \pi  \alpha_{(d_1, \ldots , d_{s-1})}  }{ F_{s-1}^{\prime} }
    \end{align*}
\end{definition}
\begin{remark}
    The groups $H^{\overline{d} }$, $H_n^{\overline{d} }$ are subgroups of $F_{s-1}/F_{s-1}^{\prime} = H_1(Y_0, \mathbb{Z}) \cong \mathbb{Z}^{s-1}$ and the matrix $R$ allows us to construct bases $B_i$ of the free module $\mathbb{Z}^{s-1}$ so that 
    \begin{align*}
        H^{\overline{d}} &=  \bigoplus_{i=2}^{s-1} B_i \mathbb{Z}
        \\
        H_n^{\overline{d}} &= n B_1 \oplus \bigoplus_{i=2}^{s-1} B_i \mathbb{Z}. 
    \end{align*}
    Namely we can take as $B_i$ the rows of the matrix $R$. 
\end{remark}

\begin{example}
    Assume that $(d_1, d_2, d_3)=(10,15,20)$ and $n=12$. We compute that the greatest common divisor $(10,15,20)=5$. The Smith normal form is computed 
    \begin{equation*}
    (10, 15, 20)
        \begin{pmatrix}
            0 & 1 & 0 \\
            -1 & 2 & 4 \\
            1 & -2 & -3
        \end{pmatrix} 
    =(5, 0 ,0).
    \end{equation*}
 Therefore, the set of solutions to congruence (\ref{eq:sol1}) is given by 
 \begin{align*}
    \begin{pmatrix}
        l_1 \\ l_2 \\  l_3
    \end{pmatrix}
     &= 
    \begin{pmatrix}
        0 & 1 & 0 \\
        -1 & 2 & 4 \\
        1 & -2 & -3
    \end{pmatrix}
    \begin{pmatrix}
        12 t_1 \\ t_2 \\ t_3
    \end{pmatrix} 
        =
        \begin{pmatrix}
            t_2 \\ -12 t_1+2 t_2 +4 t_3\\ 12 t_1 - 2 t_2 -3 t_3
        \end{pmatrix}.
  \end{align*}
The group $\ker \alpha_{(10,15,20)}$  for $d_{1}=10$, $d_2=15$, $d_3=20$ is 
normally generated by commutator words
\[
[x_i,x_j], \text{ for all } 1\leq i < j \leq 3 
\]
and words 
\[
x_2^{-12}x_3^{12}, x_{1}x_2^{2}x_3^{-2},x_2^{4}x_3^{-3}.  
\] 
\end{example}
\begin{example}
    \label{exDedicata}
  Assume that $(d_1,\ldots ,d_{s-1})=(1,1,\ldots ,1)$. Then the Smith normal form is computed as follows:
  \begin{equation*}
     \label{eq:SNFI}
    (1,1,\ldots ,1)
    \begin{pmatrix}
        1 & -1 &  \cdots & -1 & -1 \\
        0 & 1 & 0 & \cdots & 0 \\
        \vdots & \ddots & \ddots & \ddots & \vdots \\
        \vdots &  & \ddots & 1 & 0 \\
        0 & \cdots & \cdots & 0 & 1 \\ 
    \end{pmatrix}
    =(1,0,\ldots ,0).
\end{equation*}
Similarly as before the solutions to eq. (\ref{eq:sol1}) are given by  
\[
    \begin{pmatrix}
        l_1 \\ l_2 \\ \vdots \\ l_{s-1}
    \end{pmatrix}
    = 
\begin{pmatrix}
        1 & -1 &  \cdots & -1 & -1 \\
        0 & 1 & 0 & \cdots & 0 \\
        \vdots & \ddots & \ddots & \ddots & \vdots \\
        \vdots &  & \ddots & 1 & 0 \\
        0 & \cdots & \cdots & 0 & 1 \\ 
    \end{pmatrix}
    \begin{pmatrix}
    nt_1 \\ t_2 \\ \vdots \\ t_{s-1}
    \end{pmatrix}= 
    \begin{pmatrix}
    nt_1 -t_2 - \cdots - t_{s-1} \\
    t_2 \\ 
    \vdots  \\ 
    t_{s-1}
    \end{pmatrix}
\]
for $t_1, \ldots , t_{s-1} \in \mathbb{Z}$. 
The group $\mathrm{ker} \alpha_{1,\ldots,1}$ is normally generated by the commutators $[x_i,x_j]$, $1\leq i<j \leq s-1$ and the following set of generators:
\[
    x_1^n, x_1^{-1}x_j, \qquad  2 \leq j \leq s-1. 
\]
\end{example}

Motivated by example \ref{exDedicata} we have the following expression for the Smith normal form 
\begin{proposition}
\label{prop:smithNN}
Let $d$ be the greatest common divisor of the integers $(d_1, \ldots , d_{s-1}) \in \mathbb{N}^{s-1}$. 
Let $h_1, \ldots , h_{s-1}$ be 
integers such that
\[
h_1 d_1 + \cdots + h_{s-1} d_{s-1}=d
\]
and set $\delta_i=d_i/(d_1,d_i)$ and $\Delta_i=d_1/(d_1,d_i)$. Then 
    \begin{equation}
        \label{eq:DS}
    (d_1,d_2,\ldots ,d_{s-1})
    \begin{pmatrix}
        h_1 & -\delta_2 &  \cdots & -\delta_{s-2} & -\delta_{s-1} \\
        h_2 & \Delta_2 & 0 & \cdots & 0 \\
        \vdots & 0 & \ddots & \ddots & \vdots \\
        \vdots & \vdots & \ddots & \Delta_{s-2} & 0 \\
        h_{s-1}  & 0 & \cdots & 0 & \Delta_{s-1} \\     
    \end{pmatrix}
    =(d,0,\ldots ,0)
\end{equation} 
  If moreover 
  \[
    \frac{d d_1^{s-3}}{(d_1,d_2)(d_1,d_3)\cdots (d_1,d_{s-1})}=1,
  \]
  then the matrix given above is a Smith normal form. 
\end{proposition}
\begin{proof}
Observe that $d_j \Delta_j-d_1 \delta_j=0$. This proves eq. (\ref{eq:DS}). We compute the determinant of the square matrix of eq. (\ref{eq:DS}) by applying Laplace expansion along the first column, in order to obtain 
{\tiny
\[
    h_1 
    \begin{vmatrix}
    \Delta_2 & 0 & \cdots & 0 \\ 
    0 & \ddots & \ddots & \vdots \\ 
    \vdots & \ddots & \ddots & 0\\
    0 & \cdots & 0 & \Delta_{s-1}
    \end{vmatrix}
    - 
    h_2 
    \begin{vmatrix}
    -\delta_2 & -\delta_3 & \cdots & -\delta_{s-1}\\
    0  & \Delta_3 & \cdots & 0 \\ 
    \vdots & \ddots & \ddots & 0\\
    0 & \cdots & 0 & \Delta_{s-1}
    \end{vmatrix}
\]
\[
+ h_3  
    \begin{vmatrix}
    -\delta_2 & -\delta_3 & -\delta_4 & -\delta_5 & \cdots & -\delta_{s-1}\\
     \Delta_2  & 0 & 0 & \cdots& \cdots  & 0 \\ 
    0  & 0 & \Delta_4  & \ddots &  & \vdots\\
    \vdots  & \vdots & \ddots & \Delta_5 & \ddots & \vdots\\
    \vdots  & \vdots  & \cdots & \ddots  &\ddots  & 0 \\
    0 & 0 & \cdots & \cdots &  0 & \Delta_{s-1}
    \end{vmatrix} + 
    \cdots + 
    (-1)^{s-1} h_{s-1}
    \begin{vmatrix}
    -\delta_2 & -\delta_3 & \cdots & -\delta_{s-2} &  -\delta_{s-1}\\
     \Delta_2  & 0 & \cdots & 0 & 0 \\ 
    0  &  \Delta_3 & \ddots  & \vdots & \vdots \\
    \vdots  & \ddots & \ddots & 0 & \vdots & \\
    0  & \cdots  & 0 & \Delta_{s-2}  &0   \\
    \end{vmatrix}  
\]
}
\[
    = h_1 \Delta_2\cdots \Delta_{s-1} + 
    h_2 \delta_2 \Delta_3 \cdots \Delta_{s-1} 
    + h_3 \Delta_2 \delta_3 \Delta_4 \cdots \Delta_{s-1} 
    + \cdots +  h_{s-1} d_{s-1} \Delta_2 \cdots \Delta_{s-1}.
\]
In the above computation each minor determinant has been computed by using the Laplace expansion along the $i$-th column. 
Set $D=(d_1,d_2)(d_1,d_3) \cdots (d_1,d_{s-1})$. The desired determinant equals 
\[
    h_1 
    \frac{d_1^{s-2} }{D}+
    h_2 
    \frac{d_2 d_{1}^{s-3}}{D}+ 
    \cdots 
    +
    h_{s-1}
    \frac{d_{s-1} d_1^{s-3}}{D}
    =\frac{d d_1^{s-3}}{D}.
\]
The result follows. 
\end{proof}

\begin{example} $\;$ \\

\noindent $\bullet$  The numbers $(d_1,d_2,d_3)=(10,15,20)$ have $d=5$ and 
$(d_1,d_2)=5$, $(d_1,d_3)=10$, thus
$\frac{d d_1}{(d_1,d_2)(d_1,d_3)}=\frac{5 \cdot 10}{ 5 \cdot 10}=1$. Therefore, the matrix
\[
R=
\begin{pmatrix}
    h_1  & -\delta_2  & -\delta_3 \\ 
    h_2 &  \Delta _2 &  0 \\
    h_3  &  0  & \Delta _3
\end{pmatrix}
=
\begin{pmatrix}
    0  & -2  & -1 \\ 
    1 & 3 &  0 \\
    -1  &  0  & 2
\end{pmatrix}
\]
has determinant $1$ and together with the matrix $S=1$ provide  the Smith normal form. 

\noindent $\bullet$ The numbers $(d_1,d_2,d_3)=(12,9,15)$ have $d=3$ and 
$(d_1,d_2)=3$, $(d_1,d_3)=3$, thus
$\frac{d d_1}{(d_1,d_2)(d_1,d_3)}=\frac{3 \cdot 12}{ 9}=4$.
Therefore, the matrix
\[
R=
\begin{pmatrix}
    h_1  & -\delta_2  & -\delta_3 \\ 
    h_2 &  \Delta _2 &  0 \\
    h_3  &  0  & \Delta _3
\end{pmatrix}
=
\begin{pmatrix}
    1  & -3  & -5 \\ 
    -1 & 4 &  0 \\
    0  &  0  & 4
\end{pmatrix}
\]
has determinant $4$ and does provide the  Smith normal form. 
\end{example}

\section{Schreier's lemma and generators}

\label{sec:schreier}

We will employ the Reidemeister-Schreier method, 
 \cite[chap. 2 sec. 8]{bogoGrp},\cite[sec. 2.3 th. 2.7]{MagKarSol} in order to compute the groups $\mathrm{ker}(\pi \circ \alpha _{\overline{d} })$ and $\mathrm{ker}\alpha _{\overline{d} }$.
Let $F_{s-1}=\langle x_1, \cdots, x_{s-1}  \rangle$ be the free group with basis $\Sigma=\{ x_1, \cdots, x_{s-1}\}$ and let $H$ be a
subgroup of of $F_{s-1}$.

 A (right) {\bf Schreier Transversal} for $H$ in $F_{s-1}$ is a set $T=\{t_1=1, \cdots, t_n \}$ of reduced words, such that each right coset of $H$ in $F_{s-1}$ contains a unique word of $T$ (called a representative of this class) and all 
 initial segments of these words also lie in $T$.
 The condition on the initial segments means that 
 if $t_i \in T$ has the decomposition as a reduced word
$t_i=x_{i_1}^{e_1} \cdots x_{i_k}^{e_k}$ (with $i_j=1, \ldots, s-1$, $e_j= \pm 1$ and $e_j=e_{j+1}$ if $x_{i_j}=x_{i_{j+1}})$, 
\begin{equation*} 
t_i=x_{i_1}^{e_1} \cdots x_{i_k}^{e_k} \in T \Rightarrow 1, x_{i_1}^{e_1}, x_{i_1}^{e_1}x_{i_2}^{e_2},\ldots, x_{i_1}^{e_1}x_{i_2}^{e_2} \cdots x_{i_k}^{e_k} \in T.
\end{equation*}
 In particular, $1$ lies in $T$ (and represents the class $H$) and $H t_i \neq H t_j$, $\forall i \neq j$. For any $g \in F_{s-1}$ denote by $\overline{g}$ the element of $T$ with the property $Hg=H\overline{g}$.

 Notice that for any subgroup of a free group with basis $\Sigma$ there exist a (non-unique) Schreier transversal, see \cite[Th. 8.10]{bogoGrp}.

\begin{lemma}[Schreier's lemma]
\label{lemma:schreier}
Let $T$ be a right Schreier Transversal for $H$ in $F_{s-1}$ and set $\gamma(t,x):= tx \overline{tx}^{-1}$, $t \in T$, $x \in \Sigma$ and $tx \notin T$. Then $H$ is freely generated by the set 
\begin{equation} \label{free-quods}
\{ \gamma(t,x) | t\in T, x\in \Sigma, \gamma(t,x) \neq 1 \rangle
\}.
\end{equation}
\end{lemma}
It is known that the natural map $\mathrm{Aut}(F_{s-1}) \rightarrow \mathrm{GL}(s-1,\mathbb{Z}) = \mathrm{Aut}(F_{s-1}/F_{s-1}^{\prime} ) $ is an epimorphism, see \cite[ch. 3 th. 1.7]{bogoGrp}. This means that every  matrix $R \in \mathrm{SL}_{s-1}(\mathbb{Z})$ can be (non-uniquely) lifted to an automorphism $\tau_{R}$ such that  
\begin{equation}
    \label{eq:ygens}
    F_{s-1} \ni y_i=\tau_R(x_i)=
    x_1^{r_{1,i}} x_2^{r_{2,i}} \cdots x_{s-1}^{r_{s-1,i}}C_i.
\end{equation}
 where $C_i \in F_{s-1}^{\prime}$.
It is clear that 
\[
\alpha_{\overline{d} }(y_j)
=
\begin{cases}
    \displaystyle \sum_{\nu=1}^{s-1} d_\nu r_{\nu,1}=d, & \text{ if } j=1
    \\
    \displaystyle \sum_{\nu=1}^{s-1} d_\nu r_{\nu,j}=0, &\text{ otherwise}
\end{cases}
\]
\begin{remark}
The existence of the element $C_i$ is necessary. For example the matrix 
\[
    \mathrm{SL}_2(\mathbb{Z}) \ni 
    \begin{pmatrix}
        0 & 1 \\ 1 & 1
    \end{pmatrix} = \Omega 
\]
can be lifted to the automorphism $\sigma \in \mathrm{Aut}(F_2)=\langle x_1,x_2\rangle$ given by $\sigma (x_1)=x_2$, $\sigma(x_2)=x_{1}x_2$. On the other hand
\[
    \sigma^3(x_1)=\sigma^2(x_2)= \sigma (x_{1}x_2 )=x_2 x_1 x_2,
\]   
while 
\[
\Omega^3= 
\begin{pmatrix}
    1 & 2 \\ 2 & 3
\end{pmatrix},
\]
which gives an abelianized version of the above automorphism, and corresponds to the element
$\sigma^3(x_1)= x_2 x_1 x_2 = x_1 x_2 [x_2^{-1},x_1^{-1}] x_2 =x_1x_2^2 \cdot [x_2^{-1}, [x_2^{-1},x_1^{-1}]^{-1}] $. 
\end{remark}

The set $T=\{y_1^\nu: 0 \leq \nu < n\}$ is a Schreier transversal for the group $\mathrm{ker} \pi\alpha_{\overline{d} }$ with respect to the free generators $y_1,\ldots,y_{s-1}$, while $T_0=\{y_1^\nu: \nu \in \mathbb{Z} \}$ is a Schreier transversal for the group $\mathrm{ker} \alpha_{\overline{d} }$.
Schreier's lemma allows us to prove theorem \ref{th:Sch}.
Indeed,
consider first the $\mathrm{ker} \pi\alpha_{\overline{d} }$ case.  Observe that 
\[
    \overline{y_1^\nu y_j}=
    \begin{cases}
        y_1^\nu, &\text{ if }  j\neq 1;\\
        y_1^{\nu+1}, &\text{ if } j=1, \nu+1<n ;\\
        1, &\text{ if } j=1, \nu+1=n.
    \end{cases}   
\] 
The result follows by Schreier's lemma by computing 
\[
    y_1^{\nu} y_j \left(\overline{ y_1^\nu y_j}\right)^{-1} \text{ for } 0\leq \nu < n. 
\]
For the $\mathrm{ker} \alpha_{\overline{d} }$ case we have that 
\[
    \overline{y_1^\nu y_j}=
    \begin{cases}
        y_1^\nu, &\text{ if }  j\neq 1;\\
        y_1^{\nu+1}, &\text{ if } j=1.\\
    \end{cases} 
\]
The result follows by Schreier's lemma by computing 
\[
    y_1^{\nu} y_j \left(\overline{ y_1^\nu y_j}\right)^{-1}, \nu \in \mathbb{Z}. 
\]

\begin{example}
 When $(d_1, \ldots , d_s)=(1, \ldots ,1)$ we have the Smith normal form given in example \ref{exDedicata}. Then, $y_1=x_1$ while for $2 \leq j \leq s-1$ we have  $y_j=x_1^{-1}x_j$.
\end{example}


\section{Theory of $S$-graphs and folding}

\label{sec:folding}

We will present the theory of $S$-graphs for subgroups $H$ of a free group $F(S)$
in the set of free generators $S$. This theory will give us a method in order to compute the intersection of two groups. We are following the presentation of  \cite[sec. 21]{bogoGrp}. 

\begin{definition}
    A connected graph $\Gamma$ with a distinguished vertex $\gamma_0$ and set of edges $\Gamma^1$, together with a function $s:\Gamma^{1}\rightarrow S \cup S^{-1}$ called {\em labeling}, is an $S$-graph, if the labeling $s$ maps the star of any vertex of $\Gamma$ bijectively onto $S \cup S^{-1}$.  
\end{definition}

We will describe now a method that given a subgroup $H$ produces an $S$-graph.
\begin{method}
    \label{meth:Sgraph}
Let  $H$ be a subgroup of $F(S)$, freely generated by the elements $h_1, \ldots , h_t$, which are words in the generators $S \cup S^{-1}$. Let $l_\nu$ denote the length of the word $h_\nu$. We consider a graph $\Gamma_0$ with one vertex $\gamma_0$ and $t$ loops emanating from this vertex. We orient each loop and then divide the $\nu$-th loop, for each $1 \leq \nu \leq t$, into $l_\nu$ segments, and we label each segment with an element of $S \cup S^{-1}$ so that the word reading along the $\nu$-th loop is equal to $h_\nu$. 
We can reverse the orientation of an edge by changing its label from $x \in S$ to $x^{-1} \in S^{-1}$, or from $x \in S^{-1}$ to $x \in S$.  In this graph it might happen that two of the edges have the same initial vertex $v$ and the same label, contradicting the hypothesis that the labeling at the star of $v$ is a bijection. In order to remedy this we identify the two edges and their terminal vertices. This operation is called {\em folding}. We repeat this procedure until there are no edges labeled with the same letter of $S \cup S^{-1}$ and the same initial vertex. 
When there is no other folding possible, if there are vertices $v$ such that the labels of edges emanating from $v$ are missing some letter of $S \cup S^{-1}$, then we glue to these vertices an appropriate infinite subtree from the Cayley graph of the group $F(S)$. The fundamental group, in the sense of \cite[def. 16.3]{bogoGrp}, of the $S$-graph corresponding to $H$ is the group $H$ itself. 
\end{method}
\begin{theorem}
    \label{th:inter}
 Let $H_1$ and $H_2$ be subgroups of $F(S)$, and let $(\Gamma_1, \gamma_{1,0})$, 
 $(\Gamma_2, \gamma_{2,0})$ be the corresponding $S$-graphs. 
 Define the graph  $(\Gamma,\gamma_0)$  as follows: The set of vertices $\Gamma^0$ is the cartesian product of the sets of vertices of $\Gamma_1^0 \times \Gamma_2^0$, while the distinguished point $\gamma_0=(\gamma_{1,0}, \gamma _{2,0})$. 
 The set of edges $\Gamma^1$ is given by 
 \[
 \Gamma^1=\{(e_1,e_2) \in \Gamma^1_1 \times \Gamma^1_2: s(e_1)=s(e_2)\}.
 \]
 The initial vertex (resp. terminal vertex) of the edge $(e_1,e_2)$ is the product of the initial vertices (resp. terminal vertices) of the edges $e_1$ and $e_2$. The inversion of the product of two edges is the product of the inversion of the edges. Finally the label of the edge $(e_1,e_2)$ is the common label of $e_1,e_2$, that is $s(e_1,e_2)=s(e_1)=s(e_2)$. 

 Then the connected component of the graph $\Gamma$ that contains $\gamma_0$ is the $S$-graph of the group $H_1 \cap H_2$. 
\end{theorem}
\begin{proof}
    See \cite[th. 23.1, p. 102]{bogoGrp}. 
\end{proof}

\begin{remark}
According to the folding method 
if after all foldings there are vertices such that the labels of edges are less than the letters of $S \cup S^{-1}$, then by construction we glue to these vertices an infinite subtree from the Caley graph of the group $F(S)$. In what follows we omit this part of the folding construction since it does not affect the fundamental group. 
\end{remark}

We will need the following observation in order to present the desired group as an intersection of two known groups. Consider the following sequence of groups and  homomorphisms among them: 
\begin{equation} 
    \label{eq:defalpha}
\xymatrix{
    F_{s-1} 
\ar@/_1.3pc/[rrr]_{\alpha_{\overline{d} }}
\ar[r]^-{\phi} 
&  
\langle x_1^{d_1}, \ldots, x_{s-1}^{d_{s-1}} \rangle 
\ar[r]^-{i}
&
 F_{s-1} 
\ar[r]^{\alpha_{\overline{1} }}
&  \mathbb{Z} 
\ar[r]^-{\pi }
&
\frac{d \mathbb{Z}}{dn \mathbb{Z}},
}
\end{equation}
where $i$ is an inclusion and 
\[
\phi: x_j \mapsto x_j^{d_j} \, , \, 1 \leq j \leq s - 1.
\]

Then
\[
\alpha_{\overline{1} } \circ i \circ  \phi = \alpha_{\overline{d} }.
\]

\begin{proposition}
    \label{prop:inter}
\[
    \mathrm{\ker }(\alpha _{\overline{d} })= 
    \phi^{-1}
    \big(
    R_{n,s-1} \cap 
    \langle x_1^{d_1}, \ldots, x_{s-1}^{d_{s-1}} \rangle
    \big).
\]
\end{proposition}
\begin{proof}
    Recall that $R_{n,s-1}=\ker \alpha_{\overline{1} }$. We have
\begin{align*}
   w \in \ker 
   \alpha_{\overline{d} } & \Leftrightarrow
     \alpha_{\overline{1} }\circ i \circ \phi(w)=0 \Leftrightarrow 
     \phi(w) \in \ker \alpha_{\overline{1} } \cap \im \phi 
     \\
    & \Leftrightarrow w \in \phi^{-1}
     \big(
     R_{n,s-1} \cap 
     \langle x_1^{d_1}, \ldots, x_{s-1}^{d_{s-1}} \rangle
     \big).
\end{align*}

\end{proof}

In order to compute the intersection of the groups $R_{n,s-1}$ and $\langle x_1^{d_1}, \ldots, x_{s-1}^{d_{s-1}} \rangle$ we will compute their $S$-graphs and then we will apply theorem \ref{th:inter}.

\begin{lemma}
    The $S$-graph of the group $R_{n,s-1}$ is given on the left hand side of  figure \ref{fig:Rn}. It consist of a graph with $n$-vertices $y^{(1)},\ldots ,y^{(n)}$ and all group generators $x_1,\ldots ,x_{s-1}$ decorating the edges from $y^{(i)}$ to $y^{(i+1)}$. Notice that $y^{(n+1)}=y^{(1)}$.  
\end{lemma}
\begin{proof}
We will apply method \ref{meth:Sgraph} for constructing the desired $S$-graph. 
recall that 
\begin{align*}
    R_{n,s-1} &= \langle  \{ x_1^{i} x_j x_1^{-i-1}: 0 \leq i \leq n-2, 2 \leq j \leq s-1 \} \cup \{x_1^{n-1}x_j:1 \leq j  \leq s-1\} \rangle.
\end{align*}
We will prove first that the $S$-graph of the group 
\[
    G_{n,s-1} = \langle  \{ x_1^{i} x_j x_1^{-i-1}: 0 \leq i \leq n-2, 2 \leq j \leq s-1 \} \rangle
\]
is the subgraph of Figure \ref{fig:Rn} with edges in red color. We will use induction on $n$. For $n=2$, the group 
\[
G_{2,s-1} = \{ x_j x_1^{-1}:  2 \leq j \leq s-1 \} \rangle
\]
have the  $S$-graph depicted on the left hand side of Figure (\ref{fig:indG}) by definition  after folding all the common edges $x_1$. 
Assume that the $S$-graph of the group $G_{n,s-1}$ is the one depicted in the second column of Figure   (\ref{fig:indG}). We will now consider the case of the group $G_{n+1,s-1}$, which has all the generators of $G_{n,s-1}$ plus the generators 
\[
    x_1^{n+1} x_j x_1^{-n-2}, \quad 2 \leq j \leq s-1.
\] 
The inclusion of this generators gives the graph in the third column of Figure
(\ref{fig:indG}) and after repeated folding we arrive at the right column of figure (\ref{fig:indG}), finishing the induction for the group $G_{n,s-1}$.

\begin{figure}[H]
    \begin{center}
\begin{tikzpicture}
    \foreach \i in {1} {
            \draw[ middlearrow={<},red] (0, 3-2*\i) .. controls (-0.5, 2.5-2*\i) and (-0.5, 1.5-2*\i) .. (0, 1-2*\i) node[midway, anchor=south west] {$\!\!\!\!\!\!\!\!\!\!\!\!x_1$};
            \draw[ middlearrow={<},red] (0, 3-2*\i) .. controls (-0, 2.5-2*\i) and (-0, 1.5-2*\i) .. (0, 1-2*\i) node[midway, anchor=south west] {$x_2$};
            \draw[middlearrow={<},
             red] (0, 3-2*\i) .. controls (0.5, 2.5-2*\i) and (0.5, 1.5-2*\i) .. (0, 1-2*\i) node[midway, anchor=south west] {$\cdots$};
             \draw[middlearrow={<},
             red] (0, 3-2*\i) .. controls (1.2, 2.5-2*\i) and (1.2, 1.5-2*\i) .. (0, 1-2*\i) node[midway, anchor=south west] {$x_{s-1}$};
        }

        \foreach \i in { 1, 2} {
            \filldraw[blue] (0, 3-2*\i) circle (1pt);
        }
       
       \node[blue] at (0.7, 1) {$y^{(2)}$};
       \node[blue] at (0.7, -1) {$y^{(1)}$};
\end{tikzpicture}
\qquad
\scalebox{0.7}{
\begin{tikzpicture}
    \foreach \i in {1,2,3} {
            \draw[ middlearrow={<},red] (0, 3-2*\i) .. controls (-0.5, 2.5-2*\i) and (-0.5, 1.5-2*\i) .. (0, 1-2*\i) node[midway, anchor=south west] {$\!\!\!\!\!\!\!\!\!\!\!\!x_1$};
            \draw[ middlearrow={<},red] (0, 3-2*\i) .. controls (-0, 2.5-2*\i) and (-0, 1.5-2*\i) .. (0, 1-2*\i) node[midway, anchor=south west] {$x_2$};
            \draw[middlearrow={<},
             red] (0, 3-2*\i) .. controls (0.5, 2.5-2*\i) and (0.5, 1.5-2*\i) .. (0, 1-2*\i) node[midway, anchor=south west] {$\cdots$};
             \draw[middlearrow={<},
             red] (0, 3-2*\i) .. controls (1.2, 2.5-2*\i) and (1.2, 1.5-2*\i) .. (0, 1-2*\i) node[midway, anchor=south west] {$x_{s-1}$};
        }

        \foreach \i in { 1, 2, 3, 4} {
            \filldraw[blue] (0, 3-2*\i) circle (1pt);
        }
       
       \node[blue] at (0.7, 1) {$y^{(n)}$};
       \node[blue] at (0.7, -1) {$y^{(n-1)}$};
       \node[blue] at (0.7, -5) {$y^{(1)}$};
\end{tikzpicture}
}
\qquad
\scalebox{0.5}
{
\begin{tikzpicture}
    \foreach \i in {1,2,3} {
            \draw[ middlearrow={<},red] (0, 3-2*\i) .. controls (-0.5, 2.5-2*\i) and (-0.5, 1.5-2*\i) .. (0, 1-2*\i) node[midway, anchor=south west] {$\!\!\!\!\!\!\!\!\!\!\!\!x_1$};
            \draw[ middlearrow={<},red] (0, 3-2*\i) .. controls (-0, 2.5-2*\i) and (-0, 1.5-2*\i) .. (0, 1-2*\i) node[midway, anchor=south west] {$x_2$};
            \draw[middlearrow={<},
             red] (0, 3-2*\i) .. controls (0.5, 2.5-2*\i) and (0.5, 1.5-2*\i) .. (0, 1-2*\i) node[midway, anchor=south west] {$\cdots$};
             \draw[middlearrow={<},
             red] (0, 3-2*\i) .. controls (1.2, 2.5-2*\i) and (1.2, 1.5-2*\i) .. (0, 1-2*\i) node[midway, anchor=south west] {$x_{s-1}$};
        }

                \draw[middlearrow={<}] (0, 3) .. controls (-1.7, 0) and (-1.5, -3) .. (0, -5)  node[anchor=north west] {};
                \draw[middlearrow={<}] (0, 3) .. controls (-2.7, 0) and (-2.5, -3) .. (0, -5)  node[anchor=north west] {};
                \draw[middlearrow={<}] (0, 3) .. controls (-2.2, 0) and (-2, -3) .. (0, -5)  node[anchor=north west] {};
                \draw[middlearrow={<}] (0, 3) .. controls (-3.1, 0) and (-3.1, -3) .. (0, -5)  node[anchor=north west] {};

        \foreach \i in {0, 1, 2, 3, 4} {
            \filldraw[blue] (0, 3-2*\i) circle (1pt);
        }
        \node[blue] at (0.7, 3) {$y^{(n+1)}$};
       \node[blue] at (0.7, 1) {$y^{(n)}$};
       \node[blue] at (0.7, -1) {$y^{(n-1)}$};
       \node[blue] at (0.7, -5) {$y^{(1)}$};
       
        \filldraw[blue] (-0.86, 3-2) circle (1pt);
        \filldraw[blue] (-1.13, 3-2) circle (1pt);
        \filldraw[blue] (-1.375, 3-2) circle (1pt);
        \filldraw[blue] (-1.61, 3-2) circle (1pt);

        \filldraw[blue] (-1.20, 3-4) circle (1pt);
        \filldraw[blue] (-1.58, 3-4) circle (1pt);
        \filldraw[blue] (-1.95, 3-4) circle (1pt);
        \filldraw[blue] (-2.3, 3-4) circle (1pt);

        \filldraw[blue] (-0.98, 3-6) circle (1pt);
        \filldraw[blue] (-1.28, 3-6) circle (1pt);
        \filldraw[blue] (-1.60, 3-6) circle (1pt);
        \filldraw[blue] (-1.94, 3-6) circle (1pt);

       \node at (2, -1) {$n-1$};
        \node at (-1.1, 1.5) {$x_2$};
        \node at (-1.5, 1.5) {$\!\!\!\!\!\!\!\!\!x_1$};
        \node at (-0.3, 1.5) {$x_{s-1}$};
\end{tikzpicture}
}
\qquad
\scalebox{0.5}{
\begin{tikzpicture}
    \foreach \i in {0,1,2,3} {
            \draw[ middlearrow={<},red] (0, 3-2*\i) .. controls (-0.5, 2.5-2*\i) and (-0.5, 1.5-2*\i) .. (0, 1-2*\i) node[midway, anchor=south west] {$\!\!\!\!\!\!\!\!\!\!\!\!x_1$};
            \draw[ middlearrow={<},red] (0, 3-2*\i) .. controls (-0, 2.5-2*\i) and (-0, 1.5-2*\i) .. (0, 1-2*\i) node[midway, anchor=south west] {$x_2$};
            \draw[middlearrow={<},
             red] (0, 3-2*\i) .. controls (0.5, 2.5-2*\i) and (0.5, 1.5-2*\i) .. (0, 1-2*\i) node[midway, anchor=south west] {$\cdots$};
             \draw[middlearrow={<},
             red] (0, 3-2*\i) .. controls (1.2, 2.5-2*\i) and (1.2, 1.5-2*\i) .. (0, 1-2*\i) node[midway, anchor=south west] {$x_{s-1}$};
        }


        \foreach \i in { 0, 1, 2, 3, 4} {
            \filldraw[blue] (0, 3-2*\i) circle (1pt);
        }
        \node[blue] at (0.7, 3) {$y^{(n+1)}$};
       \node[blue] at (0.7, 1) {$y^{(n)}$};
       \node[blue] at (0.7, -1) {$y^{(n-1)}$};
       \node[blue] at (0.7, -5) {$y^{(1)}$};

\end{tikzpicture}
}
\end{center}

\caption{ 
    \label{fig:indG} Inductive proof for the graph of the group  $G_{n,s-1}$}
\end{figure}

In order to pass from the group $G_{n,s-1}$ to the group $R_{n,s-1}$ we have to add an extra set of generators, namely 
\[
    x_1^{n-1}x_j:1 \leq j  \leq s-1, 
\]
which give the long arrows from $y^{(n)}$ to $y^{(1)}$ depicted in black color in Figure (\ref{fig:Rn}).  
\end{proof}

\begin{figure}[H]
    \begin{center}
        \scalebox{0.5}
        {
    \begin{tikzpicture}[scale=1.5, >=stealth, node distance=2.5cm]

        \draw[middlearrow={>}] (0, 1) .. controls (-1.7, 0) and (-1.5, -3) .. (0, -5)  node[anchor=north west] {};
        \draw[middlearrow={>}] (0, 1) .. controls (-2.7, 0) and (-2.5, -3) .. (0, -5)  node[anchor=north west] {};
        \draw[middlearrow={>}] (0, 1) .. controls (-2.2, 0) and (-2, -3) .. (0, -5)  node[anchor=north west] {};
        \draw[middlearrow={>}] (0, 1) .. controls (-3.1, 0) and (-3.1, -3) .. (0, -5)  node[anchor=north west] {};

        \foreach \i in {1, 2, 3} {
            \draw[ middlearrow={<},red] (0, 3-2*\i) .. controls (-0.5, 2.5-2*\i) and (-0.5, 1.5-2*\i) .. (0, 1-2*\i) node[midway, anchor=south west] {$x_1$};
            \draw[ middlearrow={<},red] (0, 3-2*\i) .. controls (-0, 2.5-2*\i) and (-0, 1.5-2*\i) .. (0, 1-2*\i) node[midway, anchor=south west] {$x_2$};
            \draw[middlearrow={<},
             red] (0, 3-2*\i) .. controls (0.5, 2.5-2*\i) and (0.5, 1.5-2*\i) .. (0, 1-2*\i) node[midway, anchor=south west] {$\cdots$};
             \draw[middlearrow={<},
             red] (0, 3-2*\i) .. controls (1.2, 2.5-2*\i) and (1.2, 1.5-2*\i) .. (0, 1-2*\i) node[midway, anchor=south west] {$x_{s-1}$};
        }

        \node at (2, -1) {$n-1$};
        \node at (-1.5, -3.5) {$x_2$};
        \node at (-2, -3.5) {$x_1$};
        \node at (-0.9, -3.5) {$x_{s-1}$};

        \filldraw[blue] (-0.98, 3-6) circle (1pt);
        \filldraw[blue] (-1.30, 3-6) circle (1pt);
        \filldraw[blue] (-1.62, 3-6) circle (1pt);
        \filldraw[blue] (-1.97, 3-6) circle (1pt);

        \filldraw[blue] (-1.185, 3-4) circle (1pt);
        \filldraw[blue] (-1.53, 3-4) circle (1pt);
        \filldraw[blue] (-1.89, 3-4) circle (1pt);
        \filldraw[blue] (-2.24, 3-4) circle (1pt);
    

        \foreach \i in { 1, 2, 3, 4} {
            \filldraw[blue] (0, 3-2*\i) circle (1pt);
        }
       
       \node[blue] at (0.7, 1) {$y^{(n)}$};
       \node[blue] at (0.7, -1) {$y^{(n-1)}$};
       \node[blue] at (0.7, -5) {$y^{(1)}$};   
        
    \end{tikzpicture}
        }
        \qquad \qquad
        \scalebox{0.4}
        {
    \begin{tikzpicture}[scale=1.5, >=stealth, node distance=2.5cm]


        \foreach \i in { 1, 2, 3} {
            \draw[ middlearrow={<},red] (0, 3-2*\i) .. controls (-0.5, 2.5-2*\i) and (-0.5, 1.5-2*\i) .. (0, 1-2*\i) node[midway, anchor=south west] {$x_1$};
            \draw[ middlearrow={<},red] (0, 3-2*\i) .. controls (-0, 2.5-2*\i) and (-0, 1.5-2*\i) .. (0, 1-2*\i) node[midway, anchor=south west] {$x_2$};
            \draw[middlearrow={<},
             red] (0, 3-2*\i) .. controls (0.5, 2.5-2*\i) and (0.5, 1.5-2*\i) .. (0, 1-2*\i) node[midway, anchor=south west] {$\cdots$};
             \draw[middlearrow={<},
             red] (0, 3-2*\i) .. controls (1.2, 2.5-2*\i) and (1.2, 1.5-2*\i) .. (0, 1-2*\i) node[midway, anchor=south west] {$x_{s-1}$};
        }
    
        \draw[dashed] (0,2) to (0,1);
        \draw[dashed] (0,-5) to (0,-6);

    

        \foreach \i in { 1, 2, 3, 4} {
            \filldraw[blue] (0, 3-2*\i) circle (1pt);
        }
       
       \node[blue] at (0.7, 1) {$y^{(n)}$};
       \node[blue] at (0.7, -1) {$y^{(n-1)}$};
       \node[blue] at (0.7, -5) {$y^{(1)}$};   
        
    \end{tikzpicture}
        }
    \end{center}
    \caption{$S$-graph of the groups $R_{n,s-1}$  and $R_{0,s-1}$ \label{fig:Rn}}
    \end{figure}

\begin{lemma}
The $S$-graph of the group 
\[
R_{0,s-1}=\ker \alpha_{\bar{1}} =
\langle x_1^i x_j x_1^{-i-1}: i \in \mathbb{Z}, 2 \leq j \leq s-1 \rangle
\] 
is depicted in the right hand side of Figure (\ref{fig:Rn}) and is an infinite graph. 
\end{lemma}
\begin{proof}
  This can be done by induction on positive integers and by induction on negative integers, similarly to the proof for $G_{n,s-1}$.   
\end{proof}
\begin{remark}
 The $S$-graph for the group $R_{n,s-1}$ is the $S$-graph of the group $R_{0,s-1}$ modulo $n$, that is the $S$-graph of the group $R_{0,s-1}$ wrapped along a cylinder with period $n$.    
\end{remark}


\begin{lemma}
The $S$-graph of the group $\langle x_1^{d_1}, \ldots , x_{s-1}^{d_{s-1}} \rangle$ is given in figure \ref{fig:xdgrp}. It consists of a bouquet of loops $l_1,\ldots,l_{s-1}$ where the $j$-th loop is divided into $d_j$ vertices $x^{(j,1)},\ldots , x^{(j,d_j)}$. The loops have a common vertex $x^{(j,1)}$ and the vertex $x^{(j,\kappa )}$ is connected to the vertex $x^{(j,\kappa +1)}$ by the edge $x_j$. 
\end{lemma}
\begin{proof}
    This is a direct application of the method \ref{meth:Sgraph}.
\end{proof}
\begin{figure}[H]
    \begin{center}
        \scalebox{0.5}
    {
        \begin{tikzpicture}
        \begin{scope}[rotate=45]
            \draw [thick,middlearrow={>}] plot [smooth cycle] coordinates {(0,0) (1,0.5) (2,0.7) (3,0.6) (4,0) (3,-0.6) (2,-0.7) (1, -0.5)};
            
            \node[fill, circle, inner sep=1pt] at (0,0) {}; 
            \node[fill, circle, inner sep=1pt] at (1,0.5) {};
            \node[fill, circle, inner sep=1pt] at (2,0.7) {} ;
            \node[fill, circle, inner sep=1pt] at (3,0.6) {};
            \node[fill, circle, inner sep=1pt] at (4,0) {};
            \node[fill, circle, inner sep=1pt] at (3,-0.6) {};
            \node[fill, circle, inner sep=1pt] at (2,-0.7) {};
            \node[fill, circle, inner sep=1pt] at (1,-0.5) {};
        
            \node at (1.1,0.3) [anchor=south east] {$x^{(1,2)}$};
            \node at (2,0.6) [anchor=south east] {$x^{(1,3)}$};
            \node at (3,0.5) [anchor=south east] {$x^{(1,4)}$};
            \node at (1,-0.5) [anchor=north west] {$x^{(1,d_1})$};
        
            \node[red] at (0.5,-0.5)  {$x_{1}$};
            \node[red] at (0.7,0.1)  {$x_{1}$};
            \node[red] at (1.5,0.4)  {$x_{1}$};
            \node[red] at (1.6,-0.5)  {$x_{1}$};
        
        \end{scope}

        \begin{scope}[rotate=150]
            \draw [thick,middlearrow={>}] plot [smooth cycle] coordinates {(0,0) (1,0.5) (2,0.7) (3,0.6) (4,0) (3,-0.6) (2,-0.7) (1, -0.5)};
            
            \node[fill, circle, inner sep=1pt] at (0,0) {}; 
            \node[fill, circle, inner sep=1pt] at (1,0.5) {};
            \node[fill, circle, inner sep=1pt] at (2,0.7) {} ;
            \node[fill, circle, inner sep=1pt] at (3,0.6) {};
            \node[fill, circle, inner sep=1pt] at (4,0) {};
            \node[fill, circle, inner sep=1pt] at (3,-0.6) {};
            \node[fill, circle, inner sep=1pt] at (2,-0.7) {};
            \node[fill, circle, inner sep=1pt] at (1,-0.5) {};
        
            \node at (1,0.5) [anchor=north] {$x^{(2,2)}$};
            \node at (2,0.7) [anchor=north] {$x^{(2,3)}$};
            \node at (3,0.6) [anchor=north] {$x^{(2,4)}$};
            \node at (1,-0.5) [anchor=south east] {$x^{(2,d_2)}$};
        
            \node[red] at (0.5,-0.5)  {$x_{2}$};
            \node[red] at (0.3,0.5)  {$x_{2}$};
            \node[red] at (1.5,0.4)  {$x_{2}$};
            \node[red] at (1.6,-0.5) [anchor=south west] {$x_{2}$};
        
        \end{scope}
        
        \begin{scope}[rotate=-80]
            \draw [thick,middlearrow={>}] plot [smooth cycle] coordinates {(0,0) (1,0.5) (2,0.7) (3,0.6) (4,0) (3,-0.6) (2,-0.7) (1, -0.5)};
            
            \node[fill, circle, inner sep=1pt] at (0,0) {}; 
            \node[fill, circle, inner sep=1pt] at (1,0.5) {};
            \node[fill, circle, inner sep=1pt] at (2,0.7) {} ;
            \node[fill, circle, inner sep=1pt] at (3,0.6) {};
            \node[fill, circle, inner sep=1pt] at (4,0) {};
            \node[fill, circle, inner sep=1pt] at (3,-0.6) {};
            \node[fill, circle, inner sep=1pt] at (2,-0.7) {};
            \node[fill, circle, inner sep=1pt] at (1,-0.5) {};
        
            \node at (1,0.5) [anchor=north west] {$x^{(j,2)}$};
            \node at (2,0.7) [anchor=north west] {$x^{(j,3)}$};
            \node at (3,0.6) [anchor=north west] {$x^{(j,4)}$};
            \node at (1,-0.5) [anchor=east] {$x^{(j,d_j)}$};
        
            \node[red] at (0.5,-0.5)  {$x_{j}$};
            \node[red] at (0.3,0.5)  {$x_{j}$};
            \node[red] at (1.5,0.4)  {$x_{j}$};
            \node[red] at (1.6,-0.5)  {$x_{j}$};
        
        \end{scope}
         

        \end{tikzpicture}
    }

    \end{center}
    \caption{An $S$-graph for the group $\langle x_1^{d_1},\ldots ,x_{s-1}^{d_{s-1}} \rangle$ \label{fig:xdgrp} }

\end{figure}
%

 We now compute the product $S$-graph $\Gamma$ for the groups $R_{n,s-1}$ and 
$\langle x_1^{d_1}, \ldots, x_{s-1}^{d_{s-1}}  \rangle$. It consists of the vertices $(y^{(i)},x^{(j,\kappa_j)})$, $1\leq i \leq n$, $1\leq j \leq s-1$, $1 \leq \kappa_j \leq d_j$. 

From a vertex $(y^{(i)},x^{(j,\kappa_j)})$ for $ 2\leq \kappa_j \leq d_j$ emanates only one edge $x_j$ pointing to $(y^{(i+1)},x^{(j,\kappa_{j+1})})$, where $x^{(j,d_{j}+1)}=x^{(j,1)}$.
From the vertices $(y^{(i)},x^{(j,1)})$ emanate the edges $x_1,\ldots ,x_{s-1}$ pointing to $(y^{(i+1)},x^{(j,2)})$.
Start from the distinguished vertex $(y^{(1)},x^{(j,1)})$, we form a loop in the $S$ graph moving on edges with label $x^j$.
We have the following sequence of edges
\begin{equation} \label{eq:kemodn}
\xymatrix{  
    (y^{(1)},x^{(j,1)}) \ar[r]^{x_j} &
    (y^{(2)},x^{(j,2)}) \ar[r]^-{x_j} & 
    \cdots 
    \ar[r]^-{x_j} &
    (y^{(i)},x^{(j,i)}) \ar[r]^-{x_j} &
    \cdots
}
\end{equation}
It is clear that this will be a closed loop when $i=1+ k n= 1 + k^{\prime}  d_j$. This will happen the first time after the least common multiple of $n$ and $d_j$ steps. We thus form a closed loop of length $\frac{n d_j}{(n,d_j)}$ with all edges labeled by $x_j$. But this is not the only way to produce closed paths. 

Observe first that if we are on a vertex of the form 
$(y^{(i)},x^{(j,\kappa_j)})$ for $ 2\leq \kappa_j \leq d_j$ there is only one way to move, namely by edges labeled by $x_j$. We thus replace all this edges on the $S$-graph by an edge decorated by $x_j^{d_j}$ and we form a new $S$-graph $\Gamma$ with nodes $Y^{(D)}=(y^{(1+D)},x^{1})$, 
where $D$ is a $\mathbb{N}$-linear combination of $d_1, \ldots ,d_{s-1}$. The vertex $x^{1}$ is independent of the path and $D$, since $x^1=x^{(j,1)}$ for all $1 \leq j \leq s-1$. The edges of the graph $\Gamma$ are labeled by $d_j$, indicating the multiplication by $x^{d_j}$. 

If $D=d_{i_1}+d_{i_2}+ \cdots + d_{i_t}$ then we  can go from the node $Y^{(0)}$ to the node $Y^{(D)}$ by the path $x_{i_1}^{d_{i_1}}x_{i_2}^{d_{i_2}}\cdots 
x_{i_t}^{d_{i_t}}$. This means that if $D$ can be expressed in two different ways as sum of $d_1, \ldots ,d_{s-1}$, i.e. 
\[
D=d_{i_1}+d_{i_2}+ \cdots + d_{i_t} = d_{i_1^{\prime} } 
+d_{i_2 ^{\prime}} + \cdots + d_{i_{t^{\prime}} ^{\prime} }
\]
the we have the relation 
\[
x_{i_1}^{d_{i_1}}x_{i_2}^{d_{i_2}}\cdots 
x_{i_t}^{d_{i_t}} = 
x_{i_1 ^{\prime} }^{d_{i_1^\prime} }x_{i_2^{\prime} }^{d_{i_2^\prime} }\cdots 
x_{i_{t^{\prime} } ^{\prime} }^{d_{i_{t^{\prime}}^\prime }}.
\]
Notice, that this procedure is not commutative, that is the equality
\[
D=2d_1 + d_2 = d_3 \pmod n,
\]
can be interpreted by several paths joining $Y^{(0)}$ and 
$Y^{(D)}$ and  induces the relations:
\begin{equation}
    \label{eq:wordsnc}
 x_1^{d_1}x_1^{d_1}x_2^{d_2} x_3^{-d_3}=1,
 x_1^{d_1}x_2^{d_2}x_1^{d_1} x_3^{-d_3}=1,
 x_2^{d_2} x_1^{d_1}x_1^{d_1}x_3^{-d_3}=1.
\end{equation}
But we observe that since $d_i + d_j =d_j+d_i$ we always have the word $x_i^{d_i}x_j^{d_j}x_i^{-d_i}x_j^{-d_j}=[x_i^{d_i},x_j^{d_j}]$ in our group. Therefore, we need to only include one word from the set of words in eq. (\ref{eq:wordsnc}). 
\begin{figure}[H]
    \scalebox{0.7}{
\begin{tikzpicture}
    \def\drawDiagram#1#2#3#4#5#6{
        \coordinate (start) at #1;
        

        \filldraw[blue] ($(start) + (#2, 1.5)$) circle (2pt);
        
        \draw[ ->] (start) -- ++(#2, 1.5) node[anchor=south] {$\;$};
        \filldraw[blue] ($(start) + (#2, 0.5)$) circle (2pt);
        
        \draw[ ->] (start) -- ++(#2, 0.5) node[anchor=south] {$\;$};
        
        \filldraw[blue] ($(start) + (#2, -0.5)$) circle (2pt);
        
        \draw[ ->] (start) -- ++(#2, -0.5) node[anchor=south] {$\;$};
        
        \filldraw[blue] ($(start) + (#2, -1.5)$) circle (2pt);

        \draw[ ->] (start) -- ++(#2, -1.5) node[anchor=south] {$\;$};

         \node at ($(start) + (#2/2, 1.1)$) {#3};
         \node at ($(start) + (#2/2, 0.5)$) {#4};
         \node at ($(start) + (#2/2, -0.1)$) {#5};
         \node at ($(start) + (#2/2, -0.6)$) {#6};
    }

    \drawDiagram{(0,0)}{2}{$d_1$}{$d_2$}{$d_3$}{$d_4$};
    \drawDiagram{(2,1.5)}{3}{$d_1$}{$d_2$}{$d_3$}{$d_4$};
    \drawDiagram{(5,2)}{2}{$d_1$}{$d_2$}{$d_3$}{$d_4$};
    \drawDiagram{(2,-1.5)}{2}{$d_1$}{$d_2$}{$d_3$}{$d_4$};
    \drawDiagram{(4,-3)}{3}{$d_4$}{$d_2$}{$d_3$}{$d_1$};
    \drawDiagram{(7,1.5)}{3}{$d_4$}{$d_3$}{$d_2$}{$d_1$};
    \drawDiagram{(7,-1.5)}{3}{$d_4$}{$d_3$}{$d_2$}{$d_1$};

    \draw[very thick,red,
    postaction={decorate},
          decoration={markings, 
                      mark=at position 0.5 with {\arrow{>}}}
    ] (0,0) -- (2,1.5) -- (5,2) -- (7,1.5) -- (10,0) -- (7,-1.5) -- (4,-3) -- (2, -1.5) -- (0,0);
    
\end{tikzpicture}
    }

\caption{\label{fig:Sintersec} The product $S$-graph, with relation $d_1+d_2+d_3 +d_1 \equiv 4 d_4 \pmod n$ inducing the relation $x_1^{d_1}x_2^{d_2}x_3^{d_3}x_1^{d_1}x_4^{-d_4}=1$. 
}
\end{figure}

Thus the problem of finding closed paths in the graph $\Gamma$ is equivalent to the problem of finding solutions $l_1, \ldots ,l_{s-1}$ of the linear  Diophantine equation given in eq. (\ref{eq:sol1}).

\bigskip
 We will now compute the product $S$-graph for the groups $R_{0,s-1}$ and 
the group $\langle x_1^{d_1}, \ldots, x_{s-1}^{d_{s-1}}  \rangle$. In this case we can not form closed loops as we did in eq. (\ref{eq:kemodn}). We form again the product graph as in the previous case with vertices $(y^{(i)}, x^{(j, \kappa_j)})$ and arguing as before we see that the $S$-graph of the product is 
similar to the graph of $R_{0,s-1}$ as depicted on the right side of figure \ref{fig:Rn}, but each edge is decorated by $x_j^{d_j}$ instead of $x_j$.  As in the previous case the set of closed paths is determined by finding solutions 
$l_1, \ldots ,l_{s-1}$ of the linear  Diophantine equation given in eq. (\ref{eq:sol2}).

\begin{example}
Let us consider coefficients $d_1,\ldots,d_s$ so that the assumptions of proposition \ref{prop:smithNN} are satisfied. Using the notation of proposition \ref{prop:smithNN} and the Smith normal form in this case as given in eq. (\ref{eq:DS}) we define the following elements according to eq. (\ref{eq:ygens}):
\begin{align*}
\bar{y}_1 &= \phi(y_1)= x_1^{d_1 h_1} \cdots x_{s-1}^{d_{s-1} h_{s-1}} \bar{c}_1 \\
\bar{y}_j &= \phi(y_j)= x_1^{\delta_j d_1} x_j^{\Delta_j d_j} \bar{c}_j \text{ for } 2 \leq j \leq s-1,
\end{align*}
We distinguish the following $n(s-2)+1$ generators in $S$-graph product
\[
    \bar{y_1}^n, \bar{y}_1^i \bar{y}_j \bar{y}_1^{-i}, \text{ for } 2\leq j \leq s-1, 0 \leq i \leq n-1. 
\] 
Their $\phi$-preimages are 
$ y_1$ for $y_1= x_1^{h_1} \cdots x_{s-1}^{h_{s-1}}c_1$ and $y_1^i y_j y_1^{-i}$ for 
$y_j=x_1^{\delta_1} x_j^{\Delta_j} c_j$, $2\leq j \leq s-1$, $0 \leq i \leq n-1$.
These elements form a basis for the group $\mathrm{ker}\alpha_{\bar{d}}$. In figure \ref{fig:SmithProd} we show the generators of $\bar{y_1}, \bar{y_j}$ inside the product graph. The group 
$\mathrm{ker} \pi \alpha_{\bar{d}}$ has a similar presentation. 

\begin{figure}
\caption{ \label{fig:SmithProd} The loop corresponding to $\bar{y}_1$ is shown in black color and the loop corresponding to $\bar{y}_j$ is shown in red color. The loop $\bar{C}_i$ shown in green color corresponds to a commutator.}
\begin{tikzpicture}[scale=0.77, transform shape]

\useasboundingbox (-20mm, -40mm) rectangle (100mm, 50mm);

\def\n{11} 
\def\R{100} 
\coordinate (center) at (0, 0); 

\coordinate (M) at (164.25pt, -537.33pt);
\coordinate (P) at (70pt, 0);
\coordinate (Q) at (130pt, 0pt);

\foreach \i/\name in {0/A, 1/B, 2/C, 3/D, 4/E, 5/F, 6/K, 7/J, 8/G, 9/L, 10/M} {
    \pgfmathsetmacro{\angle}{360/\n * \i}
    \pgfmathsetmacro{\xcoord}{\R * cos(\angle)}
    \pgfmathsetmacro{\ycoord}{\R * sin(\angle)}
    \coordinate (\name) at (\xcoord pt, \ycoord pt);
}

\draw[->, thick, green] (A)  to [in=-45,out=45,loop, min distance=10cm]  node[midway, right] {$\bar{C}_j$} (A);



\def\n{8} 
\def\R{100} 

\foreach \i/\name in {0/AA, 1/BB, 2/CC, 3/DD, 4/EE, 5/FF, 6/KK, 7/JJ, 8/GG, 9/LL, 10/MM} {
    \pgfmathsetmacro{\angle}{360/\n * \i}
    \pgfmathsetmacro{\xcoord}{\R * cos(\angle)}
    \pgfmathsetmacro{\ycoord}{\R * sin(\angle)}
    \coordinate (\name) at (-\xcoord+200 pt, \ycoord pt);
}


\draw[color_29791,line width=0.75pt,line cap=round,line join=round] (A) -- (B);
\draw[color_29791,line width=0.75pt,line cap=round,line join=round] (A) -- (P);
\draw[color_29791,line width=0.75pt,line cap=round,line join=round] (Q) -- (A);

\draw[red,line width=0.75pt,line cap=round,line join=round] (AA) -- (BB);
\draw[red,line width=0.75pt,line cap=round,line join=round] (BB)-- (CC);
\draw[red,line width=0.75pt,line cap=round,line join=round,dotted] (CC) -- (DD);
\draw[red,line width=0.75pt,line cap=round,line join=round] (DD) -- (EE);
\draw[red,line width=0.75pt,line cap=round,line join=round] (EE) -- (FF);
\draw[red,line width=0.75pt,line cap=round,line join=round] (FF) -- (KK);
\draw[red,line width=0.75pt,line cap=round,line join=round,dotted] (KK) -- (JJ);
\draw[red,line width=0.75pt,line cap=round,line join=round] (JJ) -- (AA);

\draw[->, red, line width=0.75pt] (AA) -- ($ (AA)!0.5!(BB) $);
\draw[->, red, line width=0.75pt] (BB) -- ($ (BB)!0.5!(CC) $);
\draw[->, red, line width=0.75pt] (DD) -- ($ (DD)!0.5!(EE) $);
\draw[->, red, line width=0.75pt] (EE) -- ($ (EE)!0.5!(FF) $);
\draw[->, red, line width=0.75pt] (FF) -- ($ (FF)!0.5!(KK) $);
\draw[->, red, line width=0.75pt] (JJ) -- ($ (JJ)!0.5!(AA) $);

\node at ($ (AA)!0.5!(BB) + (11pt, 0pt) $) {${\color{red} x_1^{d_1} }$};
\node at ($ (BB)!0.5!(CC) + (0pt, 10pt) $) {${\color{red} x_1^{d_1} }$};
\node at ($ (CC)!0.5!(DD) + (1pt, 10pt) $) {${\color{red} x_1^{d_1} }$};
\node at ($ (DD)!0.5!(EE) + (12pt, 0pt) $) {${\color{red} x_1^{d_1} }$};
\node at ($ (EE)!0.5!(FF) + (16pt, 3pt) $) {${\color{red} x_j^{-d_j} }$};
\node at ($ (FF)!0.5!(KK) + (5pt, -10pt) $) {${\color{red} x_j^{-d_j} }$};
\node at ($ (KK)!0.5!(JJ) + (2pt, 10pt) $) {${\color{red} x_j^{-d_j} }$};
\node at ($ (JJ)!0.5!(AA) + (15pt, -5pt) $) {${\color{red} x_j^{-d_j} }$};

\node[fill=red, circle, scale=\dotsize] at (AA) {};
\node[fill=red, circle, scale=\dotsize] at (BB) {};
\node[fill=red, circle, scale=\dotsize] at (CC) {};
\node[fill=red, circle, scale=\dotsize] at (DD) {};
\node[fill=red, circle, scale=\dotsize] at (EE) {};
\node[fill=red, circle, scale=\dotsize] at (FF) {};
\node[fill=red, circle, scale=\dotsize] at (JJ) {};
\node[fill=red, circle, scale=\dotsize] at (KK) {};

\node at ($(BB)+(5mm,-2mm)$)  {$Y^{(d_1)}$};
\node at ($(CC)+(0mm,+3mm)$)  {$Y^{(2d_1)}$};
\node at ($(DD)+(7mm,1mm)$)  {$Y^{(3d_1)}$};
\node at ($(EE)+(8mm,3mm)$)  {$Y^{(\delta_j d_1)}$};
\node at ($(EE)+(9mm,-2mm)$)  {$=Y^{(\Delta_j d_j)}$};
\node at ($(FF)+(12mm,1mm)$)  {$Y^{((\Delta_j-1) d_j)}$};
\node at ($(KK)+(0mm,-3mm)$)  {$Y^{((\Delta_j-2) d_j)}$};
\node at ($(JJ)+(0mm,-4mm)$)  {$Y^{(-d_j)}$};

\draw[line width=0.75pt,line cap=round,line join=round,dotted] (B) -- (C);

\draw[color_29791,line width=0.75pt,line cap=round,line join=round] (C) -- (D);
\draw[line width=0.75pt,line cap=round,line join=round,dotted] (D) -- (E);


\draw[line width=0.75pt,line cap=round,line join=round,dotted] (F) -- (K);
\draw[color_29791,line width=0.75pt,line cap=round,line join=round] (G) -- (L);

\coordinate (Z) at ($(P) +(-20mm,0mm)$);
\coordinate (W) at ($(Q) +(20mm,0mm)$);

\draw[line width=0.75pt,line cap=round,line join=round,dotted] (P) -- (Z);


\draw[line width=0.75pt,line cap=round,line join=round,dotted] (W) -- (Q);

\draw[line width=0.75pt,line cap=round,line join=round,dotted] (L) -- (M);


\draw[color_29791,line width=0.75pt,line cap=round,line join=round] (A) -- (M);

\node at ($(A)+(-5mm,-3mm)$)  {$Y^{(0)}$};
\node at ($(B)+(20pt,0pt)$) {$Y^{(d_1)}$};

    \def\xoffset{20mm} 
    \def\yoffset{5mm}  

    \node at ($(C) + (8mm,5mm)$) {$Y^{(d_{1}j-d_1)}$};

\node[fill=black, circle, scale=\dotsize] at (A) {};
\node[fill=black, circle, scale=\dotsize] at (B) {};
\node[fill=black, circle, scale=\dotsize] at (C) {};
\node[fill=black, circle, scale=\dotsize] at (D) {};
\node[fill=black, circle, scale=\dotsize] at (E) {};
\node[fill=black, circle, scale=\dotsize] at (F) {};
\node[fill=black, circle, scale=\dotsize] at (G) {};
\node[fill=black, circle, scale=\dotsize] at (L) {};
\node[fill=black, circle, scale=\dotsize] at (M) {};
\node[fill=black, circle, scale=\dotsize] at (J) {};
\node[fill=black, circle, scale=\dotsize] at (K) {};
\node[fill=black, circle, scale=\dotsize] at (P) {};
\node[fill=black, circle, scale=\dotsize] at (Q) {};

\node at ($(D) + (3mm,5mm)$)  {$Y^{(d_{1}j)}$};
\node at ($(E) +(10mm,0mm)$)  {$Y^{(d_{1}h_{1})}$};
\node at ($(F) +(10mm,0mm)$)  {$Y^{(d_{1}h_{1}+d_2)}$};
\node at ($(G) +(0mm,-3mm)$) {$Y^{(id)}$};
\node at ($(J) +(-3mm,-4mm)$)  {$Y^{(id-d_{s-1})}$};
\node at ($(K) +(12mm,0mm)$)  {$Y^{(id-2d_{s-1})}$};
\node at ($(L) +(12mm,0mm)$) {$Y^{(id+d_{s-1})}$};
\node at ($(M) +(5mm,-4mm)$) {$Y^{(-d_{s-1})}$};
 \node at ($(P)+(-2mm,4mm)$) {$Y^{(d_{2}) }$};
 \node at ($(Q)+(+7mm,3mm)$) {$Y^{(-d_{2})}$};

\draw[->, color_29791, line width=0.75pt] (A) -- ($ (A)!0.5!(B) $);

\draw[->, color_29791, line width=0.75pt] (C) -- ($ (C)!0.5!(D) $);

\draw[->, color_29791, line width=0.75pt] (E) -- ($ (E)!0.5!(F) $);

\draw[->, color_29791, line width=0.75pt] (K) -- ($ (K)!0.5!(J) $);
\draw[->, color_29791, line width=0.75pt] (J) -- ($ (J)!0.5!(G) $); 
\draw[->, color_29791, line width=0.75pt] (G) -- ($ (G)!0.5!(L) $);
\draw[->, color_29791, line width=0.75pt] (M) -- ($ (M)!0.5!(A) $);
\draw[->, color_29791, line width=0.75pt] (A) -- ($ (A)!0.5!(P) $);
\draw[->, color_29791, line width=0.75pt] (Q) -- ($ (Q)!0.5!(A) $);

\draw[color_29791,line width=0.75pt,line cap=round,line join=round] (A) -- (M);
\node at ($ (M)!0.5!(A) + (-14pt, 3pt) $) {${\color{red} x_{s-1}^{d_{s-1}} }$};

\draw[color_29791,line width=0.75pt,line cap=round,line join=round] (A) -- (P);
\node at ($ (A)!0.5!(P) + (-2pt, 10pt) $) {${\color{red} x_{2}^{d_{2}} }$};

\draw[color_29791,line width=0.75pt,line cap=round,line join=round] (A) -- (B);
\node at ($ (A)!0.5!(B) + (-14pt, 3pt) $) {${\color{red}  x_{1}^{d_{1}}  }$};

\draw[color_29791,line width=0.75pt,line cap=round,line join=round] (C) -- (D);
\node at ($ (C)!0.5!(D) + (6pt, 8pt) $) {${\color{red} x_{1}^{d_{1}} }$};

\draw[color_29791,line width=0.75pt,line cap=round,line join=round] (E) -- (F);
\node at ($ (E)!0.5!(F) + (14pt, -3pt) $) {${\color{red} x_{2}^{d_{2}} }$};

\draw[color_29791,line width=0.75pt,line cap=round,line join=round] (K) -- (J);
\node at ($ (K)!0.5!(J) + (18pt, -5pt) $) {${\color{red} x_{s-1}^{d_{s-1}} }$};

\draw[color_29791,line width=0.75pt,line cap=round,line join=round] (J) -- (G);
\node at ($ (J)!0.5!(G) + (20pt, 9pt) $) {${\color{red} x_{s-1}^{d_{s-1}} }$};

\draw[color_29791,line width=0.75pt,line cap=round,line join=round] (G) -- (L);
\node at ($ (G)!0.5!(L) + (-6pt, 12pt) $) {${\color{red} x_{s-1}^{d_{s-1}}  }$};

\node at ($ (Q)!0.5!(A) + (5pt, 10pt) $) {${\color{red} x_{2}^{d_{2}} }$};






\end{tikzpicture}

\end{figure}

\end{example}

\section{Braid group actions}
\label{sec:braid}
It is known that the braid group can be realized as an automorphism group of the free group. The braid group on $s-1$ strands can be generated by the elements $\sigma_i \in \mathrm{Aut}(F_{s-1})$ for $1 \leq i, j \leq s-2$, where
\[
    \sigma_i(x_j)=
    \begin{cases}
        x_j, &\text{ if }  j \neq i, i+1\\
        x_i, &\text{ if }  j=i+1;\\
        x_i x_{i+1} x_i^{-1}, &\text{ if } j=i .
    \end{cases}
\]
When $\overline{d}=\overline{1}$ there is an action of the braid group on 
$ \mathrm{\ker } \alpha _{\bar{1}}$, which gives rise to the Burau representation, see
\cite{MR4117575}. 
In general there is no topological reason that for the braid group to preserve $\mathrm{ker}\alpha_{\bar{d}}$, that is $\sigma(\mathrm{ker}\alpha_{\bar{d}})\subset \mathrm{ker}\alpha_{\bar{d}}$. In this section we will investigate when this happens. 
The action of automorphisms of the free group on elements of the groups 
 $ \mathrm{\ker } \alpha_{\bar{1}}$ is complicated in the general case of $\overline{d}$ and can simplified if we consider the action on the groups $H^{\overline{d} }$ and $H^{\overline{d} }_n$ as defined in definition \ref{def:H}. 

 Let $R$ be the matrix defined in eq. (\ref{eq:dSN}). By proposition \ref{prop:paramDioph} an element in $\mathrm{\ker} \pi \alpha_{\overline{d} }/F_{s-1}^{\prime} $ and in  $\mathrm{\ker} \alpha_{\overline{d} }/F_{s-1}^{\prime}  $  is parametrized by 
 \[
    R
\begin{pmatrix}
    nt_1 \\ t_2 \\ \vdots \\ t_{s-1}
\end{pmatrix}
\text{ and }
R
\begin{pmatrix}
    0 \\ t_2 \\ \vdots \\ t_{s-1}
\end{pmatrix}
\text{ respectively}.
\] 
The braid group element $\sigma_i$ is acting in terms of this abelianized  
setting in terms of the matrix $\mathbb{I}_{i,i+1}$, given by swaping the $i$-th and $i+1$-th columns of the identity matrix, that is 
 \[
   \mathbb{I}_{i,i+1}   R
\begin{pmatrix}
    nt_1 \\ t_2 \\ \vdots \\ t_{s-1}
\end{pmatrix}
\text{ and }
\mathbb{I}_{i,i+1}  R
\begin{pmatrix}
    0 \\ t_2 \\ \vdots \\ t_{s-1}
\end{pmatrix}
\text{ respectively}.
\] 
We may now ask if the last element is still an element in $\mathrm{\ker} \pi \alpha_{\overline{d} }/F_{s-1}^{\prime} $ and in  $\mathrm{\ker} \alpha_{\overline{d} }/F_{s-1}^{\prime}  $  respectively, that is if there are elements $t_1^{\prime} , \ldots , t_{s-1}^{\prime} $ such that
\[
    \mathbb{I}_{i,i+1}R
    \begin{pmatrix}
        nt_1 \\ t_2 \\ \vdots \\ t_{s-1}
    \end{pmatrix}
    =
    R \begin{pmatrix}
        nt_1^{\prime}  \\ t_2^{\prime}  \\ \vdots \\ t_{s-1}^{\prime} 
    \end{pmatrix}
    \text{ and }
    \mathbb{I}_{i,i+1}  R
    \begin{pmatrix}
        0 \\ t_2 \\ \vdots \\ t_{s-1}
    \end{pmatrix}
    =
    R \begin{pmatrix}
        nt_1^{\prime}  \\ t_2^{\prime}  \\ \vdots \\ t_{s-1}^{\prime} 
    \end{pmatrix}
    \text{ respectively}.  
\]
It is clear that in the $\mathrm{ker}(\alpha_{\overline{d} })$ this can be done if and only if 
\begin{equation}
    \label{eq:condBraidI}
R^{-1} \mathbb{I}_{i,i+1}R = 
\begin{pmatrix}
a_{11} & 0 & \cdots & 0 \\
a_{21} & \cdots & \cdots & a_{2,s-1} \\
\vdots &   &  & \vdots \\ 
a_{s-1,1} & \cdots & \cdots & a_{s-1,s-1}   
\end{pmatrix}
\end{equation}
while in the $\mathrm{ker}( \pi \alpha_{\overline{d} })$ this can be done if and only if 
\begin{equation}
    \label{eq:condBraidII}
R^{-1} \mathbb{I}_{i,i+1}R = 
\begin{pmatrix}
a_{11} & n \nu _2 & \cdots & n \nu_{s-1} \\
a_{21} & \cdots & \cdots & a_{2,s-1} \\
\vdots &   &  & \vdots \\ 
a_{s-1,1} & \cdots & \cdots & a_{s-1,s-1}   
\end{pmatrix}
\end{equation}
for some integers $\nu_2, \ldots , \nu_{s-1}$. Indeed, write 
$R^{-1} \mathbb{I}_{i,i+1}R= (a_{ij} )$.  The $i$-th column $e_i$ of the identity matrix 
when multiplied with $(a_{ij})$ gives rise to the element $a_{1i}$ which should be divisible by $n$ in the $\mathrm{ker}(\pi \alpha _{\overline{d} })$ case and zero in the $\mathrm{ker}(\alpha _{\overline{d} })$ case.

For example for the matrix given in eq. (\ref{eq:SNFI}) we observe that the conditions of equations (\ref{eq:condBraidI}) and (\ref{eq:condBraidII}) are satisfied and the braid group acts on the groups $H^{\overline{1} }$ and $H^{\overline{1} }_n$.

\section{Compactification}
\label{sec:compact}
The genus $g_{\overline{d}_n }$ of the complete curve $X_{n,\overline{d}}$ defined in eq. (\ref{eq:def-curve-n}) is given in 
terms of the Riemann-Hurwitz formula
\begin{align}
    2g_{X_n,\overline{d}} &= 2-2n + \sum_{P \in X_{n,\overline{d}} } (e_{P}-1 )
\nonumber
    \\
    \nonumber
    &= 2(1-n) +\sum_{i=1}^{s} \left(\frac{n}{(n,d_i)}-1\right) (n,d_i)
    \\
    &= 2+ (s-2) n - \sum_{i=1}^{s} (n,d_i).
    \label{eq:2gform}
\end{align}
We have used that in the Kummer covering $X_{n,\bar{d}} \rightarrow \mathbb{P}^1$ under the assumptions made in eq. (\ref{eq:def-curve-n}),(\ref{eq:degree0})  only the places $P_{x=b_i}$ are ramified with ramification indices $\frac{n}{(n,d_i)}$ and that the projective line has genus $0$, \cite[p. 667]{Ko:99}.  
The open curve $X_{n,\overline{d}}^0$ has fundamental group with a presentation
\begin{equation}
    \label{eq:presG}
    \pi_1(X_{n,\overline{d}}^0,y_0) \cong 
    \langle a_1,b_1, \ldots , a_{g},b_g, \gamma_1, \ldots , \gamma_r :  \gamma_1\cdots\gamma_r[a_1,b_1]\cdots [a_{g},b_g]=1\rangle,
\end{equation}
where $\gamma_1, \ldots , \gamma_r$ are small circles surrounding each branch point of 
$ X_{n,\overline{d}}$. The number $r$ is the total number of branch points of $X_{n,\bar{d}}$ and 
equals 
\begin{equation}
    \label{eq:rdef}
    r= \sum_{i=1}^{s} (n,d_i). 
\end{equation}
Therefore,  by eq. (\ref{eq:2gform}), (\ref{eq:presG}), (\ref{eq:rdef}) we have that  $\pi_1(X_{n,\overline{d}}^0,y_0)$ is a free group in 
$(s-2)n+1$ generators.  

As in \cite[sec. 5.1]{MR4117575} the cyclic group $\mathrm{Gal}(X / \mathbb{P}^1)= \langle \sigma \rangle$ acts on the group $\mathrm{ker}(\pi \circ\alpha)$ by conjugation and the elements $\gamma_{1}, \ldots, \gamma_{r}$ are small circles around each branch point, that is the elements $x_{i}^{e_i}$, $1\leq i \leq s-1$.  Let $\Gamma = \langle x_{1}^{e_1}, \ldots, x_{s}^{e_s} |x_1x_2\cdots x_{s-1} x_s=1 \rangle$. In order to compute the fundamental group of the complete curve 
\[
    R=  
    \langle a_1,b_1, \ldots , a_{g},b_g, \gamma_1, \ldots , \gamma_r :  [a_1,b_1]\cdots [a_{g},b_g]=1\rangle,
\]
we have to compute the quotient $R=\frac{R_{0}}{\Gamma \cap R_{0}}= \frac{R_{0}\cdot \Gamma}{\Gamma}$, where $R_{0}=\mathrm{ker} \pi \circ \alpha$. Indeed, we can consider the open connected set $U$ consisting by the union  of open discs covering each missing point in $X^0$ and connected by a thick paths in $X$, see (\ref{fig:SVC}). The closed curve $X=X^0 \cup U$. For a point $x_0 \in X_0 \cap U$ we have 
$\pi_1(X^0,x_0) \cong R_0$, $\pi_1(U,x_0)=\{1\}$, $\pi_1(U\cap X^0,x_0)=\Gamma$.  By Seifert van Kampen theorem the fundamental group of $X$ is the amalgam $R_0 *_\Gamma \{1\}$, where the inclusion $U \cap X^0 \rightarrow X^0$ induces $\Gamma \rightarrow R_0\cap \Gamma$ and  the inclusion $U \cap X^0 \rightarrow  U$ induces the trivial map $\Gamma\rightarrow  \{1\}$. Since $R_0 \cap \Gamma$ is a normal subgroup of $R_0$ the later group equals $R_0/R_0\cap\Gamma$, see e.g. \cite[Chap. 2, sec. 11]{bogoGrp}. 
\begin{figure}
\begin{tikzpicture}[scale=0.9]

\coordinate (x1) at (0,2);
\coordinate (y1) at (0.4,1.48);
\coordinate (x2) at (0,-1);
\coordinate (x3) at (3,1);

\draw[line width=1pt] (x1) -- (1.2,0.4);
\draw[line width=1pt] ($(x1)+(0.2,0)$) -- (1.4,0.4);
\draw[line width=1pt] (x2) -- (1.21,0.41);
\draw[line width=1pt] ($(x2)+(0.2,0.1)$) -- (1.40,0.2);
\draw[line width=1pt] (3,1) -- (1.4,0.2);
\draw[line width=1pt] ($(x3)+(-0.3,-0.025)$) -- (1.39,0.4);

\foreach \p in {x1,x2,x3} {
  \fill[white] (\p) circle (0.69);
  \draw[line width=1pt] (\p) circle (0.7);
}

\foreach \p in {x1,x2,x3} {
  \draw[red,dashed,line width=1pt] (\p) circle (0.5);
}

\foreach \p in {x1,x2,x3} {
  \fill (\p) circle (2pt);
}

\node at (x1) [left] {$x_1$};
\node at (x2) [left] {$x_2$};
\node at (x3) [right] {$x_3$};

\node[red] at ($(x1)+(0.2,0.3)$) {$\gamma_1$};
\node[red] at ($(x2)+(0.3,0.0)$) {$\gamma_2$};
\node[red] at ($(x3)+(0.3,0.2)$) {$\gamma_3$};

\draw[white,line width=3pt] (0.415,1.46) -- ($(y1)+(0.115,0.074)$);
\draw[white,line width=3.1pt] (0.45,-0.49) -- (0.54,-0.58);
\draw[white,line width=3.1pt] (2.335,0.803) -- (2.378,0.7);
\draw[line width=1pt] (-1,-2) rectangle (4,3);

\node at (3.5,2.5) {$X$};

\end{tikzpicture}

\caption{ \label{fig:SVC} Seifert Van Kampen Theorem for proving $R_0/R_0 \cap \Gamma$}
\end{figure}

Notice that $\alpha(x_{i}^{e_i})=e_id_i= \frac{n d_i}{(n,d_i)} \equiv 0 \pmod n$, therefore $\Gamma \subset R_{0}$ and $R=\frac{R_{0}}{\Gamma}$. We have the following sequence of groups
\[
    1 \longrightarrow R=\frac{R_{0}}{\Gamma} 
    \longrightarrow 
    G=\frac{F_{s-1}}{\Gamma}
 \stackrel{\psi}{\longrightarrow } \frac{F_{s-1} }{R_{0}} \rightarrow  1.
\]
We will use the theory of Alexander modules and the Crowell exact sequence, as described in Chapter $9$ from 
\cite{Morishita2011-yw}, to describe the homology $H_1(X, \Z)$.
The map $\psi$ is the quotient map
\[
\psi : F_{s-1} / \Gamma \rightarrow F_{s-1} / R_{0} \cong \mathrm{Gal}(X / \mathbb{P}^1) =:C \cong \frac{\Z}{n\Z}.
\] 
Set also $\varepsilon: \Z[C] \rightarrow \Z$ to be the augmentation map $\sum a_g g \mapsto \sum a_g$. 

We consider $\mathcal{A}_{\psi}$ to be the {\em Alexander module}, a free $\Z$-module

\[\mathcal{A}_\psi = \left( \bigoplus\limits_{g \in F_{s-1}/\Gamma} \Z[C] dg \right) / \langle d(g_1g_2) - dg_1 -\psi(g_1)dg_2: \ g_1,g_2 \in F_{s-1}/\Gamma \rangle_{\Z[C]} \]
 where $\langle \cdots \rangle_{\Z[C]}$ is considered to be the $\Z[C]$-module generated by the elements appearing inside. 

By the above definitions, ${R}_{0}^{ab}$ is $H_1(X,\Z)$. Define the map $\theta_1  :{R}_{0}^{ab} \rightarrow \mathcal{A}_\psi$ given by
\[
{R}_{0}^{ab} \ni n \mapsto dn 
\] 
and the map $\theta_2 : \mathcal{A}_\psi \rightarrow \Z[C]$ to be the homomorphism induced by 

\[
dg\mapsto \psi(g)-1 \ \textrm{ for } g \in G.
\] 
The  Crowell exact sequence of $\Z[C]$-modules \cite[sec. 9.2]{Morishita2011-yw} is given
\begin{equation}\label{Crowell}
\centering
\begin{tikzcd}
  1 \arrow[r] & {{R}_{0}^{ab} = H_1(X,\Z)} \arrow[r, "\theta_1"] & \mathcal{A}_\psi \arrow[r, "\theta_2"] & {\Z[C]} \arrow[r, 
  "\varepsilon"] & \Z \arrow[r] & 1.
  \end{tikzcd}
\end{equation}
Consider the group $G$ admitting the presentation 
\[
    G=\left\langle
    x_1,\ldots,x_{s} |x_1^{e_1}=\cdots =x_{s}^{e_s}=x_1\cdots x_s=1 
    \right\rangle.
\] 
and denote by $\pii$ is the natural epimorphism $\pii:F_s \rightarrow G$ defined by the presentation. 
Set 
$\psi \pii (x_i)=\sigma^{d_i}$
 and 
 $\Sigma_i=1+\sigma^{d_i}+\cdots+ (\sigma^{d_i})^{e_{i}-1}$. 
\begin{proposition} The module $\mathcal{A}_\psi$ admits a free resolution as a $\Z[C]$-module: 
    \begin{equation}\label{resolution}
     \begin{tikzcd}
       {\Z[C]^{s+1} } \arrow[r, "Q"] & { \Z[C]^s} \arrow[r] & \mathcal{A}_\psi \arrow[r] & 0
       \end{tikzcd}
    \end{equation} where $s+1$ and $s$ appear as the number of relations and generators of $G$ respectively. 
    The map $Q$ is 
   expressed in form of Fox derivatives \cite[sec. 3.1]{BirmanBraids},\cite[chap. 8]{Morishita2011-yw},  as follows 
   \begin{align*}
Q &=
   \begin{pmatrix}
   \psi \pii \left(\frac{\partial x_1^{e_1}}{\partial x_1}\right) & 
   \psi \pii \left(\frac{\partial x_2^{e_2}}{\partial x_1}\right) &
   \cdots
   \psi \pii \left(\frac{\partial x_s^{e_s}}{\partial x_1}\right) &
   \psi \pii \left(\frac{\partial x_1\cdots x_s}{\partial x_1}\right) \\
   \psi \pii \left(\frac{\partial x_1^{e_1}}{\partial x_2}\right) & 
   \psi \pii \left(\frac{\partial x_2^{e_2}}{\partial x_2}\right) &
   \cdots
   \psi \pii \left(\frac{\partial x_s^{e_s}}{\partial x_2}\right) &
   \psi \pii \left(\frac{\partial x_1\cdots x_s}{\partial x_2}\right) \\
   \vdots & \vdots &  & \vdots  \\
   \psi \pii \left(\frac{\partial x_1^{e_1}}{\partial x_s}\right) & 
   \psi \pii \left(\frac{\partial x_2^{e_2}}{\partial x_s}\right) &
   \cdots
   \psi \pii \left(\frac{\partial x_s^{e_s}}{\partial x_s}\right) &
   \psi \pii \left(\frac{\partial x_1\cdots x_s}{\partial x_s}\right)
   \end{pmatrix}
   \\
   &=
   \begin{pmatrix}
   \Sigma_1 &  0 & \cdots & 0 & 1 \\
0 & \Sigma_2  & \ddots & \vdots & \bar{x}_1 \\
\vdots & \ddots    & \ddots  & 0 &  \vdots \\ 
0 & \cdots  & 0 & \Sigma_s & \bar{x}_1 \bar{x}_2 \cdots \bar{x}_{s-1}
\end{pmatrix}
\end{align*}
\end{proposition}
\begin{proof}
    See  \cite[cor. 9.6]{Morishita2011-yw} and  \cite[eq. (34)]{MR4186523} for the explicit computation of the matrix $Q$. 
\end{proof}
Let $\beta_1,\ldots,\beta_{s+1} \in \Z[C]$.
We compute
\begin{equation} \label{rankd2}
\begin{pmatrix}
\Sigma_1 &  0 & \cdots & 0 & 1 \\
0 & \Sigma_2  & \ddots & \vdots & \sigma^{d_1} \\
\vdots & \ddots    & \ddots  & 0 &  \vdots \\ 
0 & \cdots  & 0 & \Sigma_s & \sigma^{d_1} \sigma^{d_2} \cdots \sigma^{d_{s-1}}
\end{pmatrix}
\begin{pmatrix}
\beta_1 \\ \vdots \\ \beta_{s+1}
\end{pmatrix}=
\begin{pmatrix}
\Sigma_1 \beta_1 + \beta_{s+1} \\
\Sigma_2 \beta_2 + \sigma^{d_1}\beta_{s+1} \\
\vdots\\
\Sigma_s \beta_s + \sigma^{d_1}\cdots \sigma^{d_{s-1}}\beta_{s+1}
\end{pmatrix}.
\end{equation}

  Observe that the element $\sigma^{d_i}$ has order $\epsilon_i=n/(n,d_i)$.
\[
\Sigma_i = \sum_{\nu=0}^{e_i-1} \sigma^{\nu d_i} 
\]
For every integer $\kappa$ we have 
\[ 
\sigma^{d_i}\Sigma_i  =  \Sigma_{i}.
\]
Using eq. (\ref{rankd2}) we see that 
the image of the map $Q$ is  the sum $A+B$, where  $A$ is the $\mathbb{Z}[C]$-submodule of $\mathbb{Z}[C]^s$ generated by the elements
$(\Sigma_{1} \beta_{1}, \Sigma_2 \beta_2, \ldots, \Sigma_s \beta_s)$. 
and $B$ contains expressions of the form  
$\beta_{s+1} (1, \sigma^{d_1}, \sigma^{d_1+d_2}, \ldots, \sigma^{d_1+ \cdots +d_{s-1}})$. 
We will now show that elements in the intersection $A \cap B$ should be of the form $(\beta,\ldots,\beta)^t$, where 
\[
    \beta= \sum_{\nu=0}^{n-1} a_\nu \sigma^\nu,
\] 
with $a_{(\nu-d_i)\mod n}=a_{\nu}$ for all $1 \leq i \leq s$. This implies in turn that 
$a_{(\nu-\delta)\mod n}=a_{\nu}$ for the greatest common divisor $\delta=(d_1,\ldots,d_s)$.

Indeed, an element  $A\cap B$ should satisfy
\[
    \begin{pmatrix}
        \Sigma_1 \beta_1  \\
        \Sigma_2 \beta_2  \\
        \vdots\\
        \Sigma_s \beta_s
        \end{pmatrix}
        =
        \begin{pmatrix}
             \beta_{s+1} \\
             \sigma^{d_1}\beta_{s+1} \\
            \vdots\\
             \sigma^{d_1+ \cdots + d_{s-1}} \beta_{s+1}
            \end{pmatrix}.
\]
By comparing the first coordinate we see that  $\beta_{s+1}$ is $\sigma^{d_1}$ invariant. Thus $\sigma^{d_1}\beta_{s+1} = \beta_{s+1}$ in the second coordinate and is also $\sigma^{d_2}$-invariant. We continue this way all the way down in order to have that the element in the intersection is 
$(\beta,\ldots, \beta)^t$ for element $\beta=\beta_{s+1}$, which is $\sigma^{d_\nu}$ invariant for all $1 \leq \nu \leq s$. 
\begin{remark}
The module $A\cap B$ is one dimensional since it is isomorphic to a submodule of $k[C]$ that is invariant under all elements $\sigma^{d_i}$ and hence under all elements $\sigma^{(n,d_i)}$ and  $\delta=(d_1,\ldots,d_n)$ is prime to $n$ by assumption. This computation  for $\dim A\cap B=1$ fits well with equation (\ref{eq:presG}).  
\end{remark}

In order to compute $\mathrm{Im}(Q)$ as a Galois module we consider the short exact sequence
\begin{equation}
\label{eq:inter}
0 \longrightarrow  A \cap B  \stackrel{i}{\longrightarrow} A \oplus B \rightarrow  A + B \rightarrow 0,
\end{equation}
where $i(x)=(x,-x)$.

\begin{proposition}
Let $k$ be a field of characteristic $p$, $(p,n)=1$. We consider now the structure of $H_1(X,k)=H_1(X,\Z)\otimes_{\Z}k$.

    \label{prop:struc}
For $1 \leq i \leq s$ we defined 
\[
    \Sigma_i=\sum_{\nu=0}^{e_i-1} \sigma^{\nu d_i}.
\]
\begin{enumerate}
    \item
The $k[C]$-module $\Sigma_i k[C]$ admits the following set as a basis 
\[
    \left\{\Sigma_i\sigma^\kappa: 0\leq \kappa < (n,d_i)\right\}
\]
and has dimension $(n,d_i)$.
\item
The $k[C]$-module $\Sigma_i k[C]$ contains the representation $\chi_\mu$ as direct summand if and only if $ n \mid \mu (n,d_i)$, i.e. 
\[
    \chi_i = \sum_{\mu=0}^{e_i-1} 
    \chi_{(n,d_i) \mu}
\]
\item  
The modules $\Sigma_i \mathbb{Z}[C]$ are isomorphic to $\mathrm{Ind}_{\Z[\langle \sigma^{d_i} \rangle]}^{\Z[C]}\Z$. 
\end{enumerate}
\end{proposition}
\begin{proof}
Observe that $\sigma^{d_i}$ generates a subgroup $H$ of $C=\langle \sigma \rangle$ of order $e_i=n/(n,d_i)$. 
For every $0\leq \mu <n$ we compute  $\sigma^{\mu t} \Sigma_i=\Sigma_i$, that is elements in the subgroup $\langle \sigma^{r_i} \rangle$ keep $\Sigma_i$ invariant.

A $k$-basis for $k[C]$ seen as a $k[C]$-module is given by $\{\sigma^i: 0 \leq i <n\}$. 
After multiplication by $\Sigma_i$ we have $\Sigma_i \sigma^m = \Sigma_i \sigma^{m'}$ if and only if $\sigma^{m-m'} \in \langle \sigma^{d_i} \rangle$.
The least integer $0<\lambda<n$ such that $\sigma^\lambda$ is a generator of $H$ is $(n,d_i)$.
Thus 
$\{\Sigma_i\sigma^{\kappa}:0\leq \kappa < (n,d_i)\}$ form a basis of the $k[C]$-module $\Sigma(t)k[C]$. 

The character $\chi(i)$ of the $k[C]$-module $\Sigma_i k[C]$ is given by 
\[
    \chi(i)(\sigma^\mu)= 
    \begin{cases}
        (n,d_i) & \text{ if } (n,d_i) \mid \mu \\ 
        0 & \text{ if } (n,d_i) \nmid \mu
    \end{cases}
\]
For the irreducible character $\chi_\mu$ we compute 
\begin{align*}
\langle \chi(i), \chi_\mu \rangle &=
\frac{1}{n} \sum_{\nu=0}^{n-1} \chi(i)(\sigma) \zeta_n^{- \mu \nu}
=
\frac{1}{n} (n,d_i) \sum_{\nu=0 \atop (n,d_i) \mid \nu}^{n-1}  \zeta_n^{- \mu \nu}
\stackrel{\nu= (n,d_i)\nu'}{=}
 \\ 
 &= 
 \frac{1}{n} (n,d_i) \sum_{\nu'=0}^{e_i-1}  \zeta_n^{- \mu (n,d_i) \nu'} =
\begin{cases}
    1 & \text{ if } n \mid \mu \cdot(n,d_i) \\  
    0 & \text{ if } n \nmid \mu \cdot(n,d_i)
\end{cases} 
\end{align*}
The equality 
$\Sigma_i \mathbb{Z}[C]$ are isomorphic to $\mathrm{Ind}_{\Z[\langle \sigma^{d_i} \rangle]}^{\Z[C]}\Z$ follows by lemma (\ref{lemma:Cyclic}).
\end{proof} 
\begin{lemma}
    \label{lemma:Cyclic}
Let $C = \langle \sigma \rangle$ be a cyclic group of order $n$. Let $H$ be a subgroup of $C$.
Let $S$ be the sum of all elements in $H$:
$$
S = \sum_{h \in H} h \in \mathbb{Z}[C].
$$
Then $S \mathbb{Z}[C]\cong \mathrm{Ind}_H^C(\mathbb{Z}_H)$.
\end{lemma}
\begin{proof}
The trivial $\mathbb{Z}[H]$-module $\mathbb{Z}_H$ is the ring of integers $\mathbb{Z}$ with trivial $H$-action. The induced module 
$$
M = \mathrm{Ind}_H^C(\mathbb{Z}_H) = \mathbb{Z}[C] \otimes_{\mathbb{Z}[H]} \mathbb{Z}
$$
and is isomorphic 
 to the quotient of the group ring $\mathbb{Z}[C]$ by the ideal generated by the relations imposed by the trivial $\mathbb{Z}[H]$-action.
$$
\mathrm{Ind}_H^C(\mathbb{Z}_H) \cong \mathbb{Z}[C] / \mathcal{I}, 
$$
where  $\mathcal{I}$, is the left ideal generated by $\{h-1\}_{h \in H}$:
$$
\mathcal{I} = \mathbb{Z}[C] \cdot \{ h-1 \mid h \in H \}.
$$
Indeed, the  tensor product is subject to the relation $x \otimes (h \cdot z) = (x h) \otimes z$ for $x \in \mathbb{Z}[C]$, $h \in H$, and $z \in \mathbb{Z}$.
Since $h \cdot z = z$, the relation becomes $x \otimes z = (x h) \otimes z$, or:
$$
(x h - x) \otimes z = (x (h-1)) \otimes z = 0.
$$
Consider the $\mathbb{Z}[C]$-module homomorphism $\Phi$:
$$
\Phi: \mathbb{Z}[C] \to S \mathbb{Z}[C], \text{ defined by }
\Phi(x) = S x.
$$
The homomorphism $\Phi$ is surjective, thus 
$
\mathbb{Z}[C] / \Ker(\Phi) \cong S \mathbb{Z}[C]
$.
The equality  $\Ker(\Phi) = \mathcal{I}$
is a standard result in the theory of group rings for cyclic groups over $\mathbb{Z}$: the annihilator of $S$ is precisely the ideal generated by the elements $\{h-1\}_{h \in H}$.
$$
\mathrm{ker}(\Phi) = \mathrm{Ann}_{\mathbb{Z}[C]}(S) = \mathbb{Z}[C] \cdot \{ h-1 \mid h \in H \} = \mathcal{I}.
$$
\end{proof}
\begin{remark}
Part (2) of proposition \ref{prop:struc} can also be proved by Frobenius reciprocity using part (3):
\[
    \langle \chi(i) ,\chi_\nu \rangle_{C} = 
    \langle 1, \mathrm{Rest}_{C_{e_i}} {\chi_\nu} \rangle_{\langle \sigma^d \rangle}  
    =\frac{1}{e_i}\sum_{\mu=0}^{e_i-1} \zeta_{n}^{r_i \nu\mu}
    =\frac{1}{e_i}\sum_{\mu=0}^{e_i-1} \zeta_{e_i}^{\nu\mu}=
    \begin{cases}
         1 & \text{ if } e_i \mid \nu  \\
         0 & \text{otherwise}    
         \end{cases}.
\]
\end{remark}

We have expressed  $H_1(X,\Z)$ in terms of the exact sequences given in eq. (\ref{Crowell}) together with eq. (\ref{resolution}) and eq. (\ref{eq:inter}). 
 Unfortunately the theory of integral representations, that is the study of the $\Z[C]$-module structure is quite subtle even for cyclic groups, see \cite{MR0140575},\cite{MR0144980} and in general computations with modules fitting in exact sequences are not straightforward.

 But when considering the module structure over a field $k$ of characteristic $p$, $(p,n)=1$, Maschke's theorem guaranties that all short exact split and thus the representation ring equals the Grothendieck ring.  We thus will study $H_1(X,\mathbb{C})= H_1(X,\mathbb{Z})\otimes_{\mathbb{Z}} \mathbb{C}$
  and arrive at the following result:

 \begin{proposition}
    \label{prop:35}
    Let $C=\mathrm{Gal}(X/\mathbb{P}^1)= \langle \sigma \rangle$ be the cyclic Galois group of order $n$.
Denote by $\chi_{\nu}$ the character of $G$ such that 
$\chi_\nu(\sigma)=\zeta_n^\nu$. 
the $\mathbb{C}[C]$-module structure of $H_1(X,\mathbb{C})$ is given by 
\[
    H_1(X,\mathbb{C})= 
    \bigoplus_{\nu=0}^{n-1} M_\nu \chi_{\nu},
\] 
where 
\begin{equation}
\label{eq:Hstruct} 
M_\nu = 
\begin{cases}
    0 &\text{ if } \nu=0 \\ 
\#\{0\leq i \leq s: n \nmid \nu  (n,d_i) \} -2 &\text{ if } \nu \neq 0
\end{cases}
\end{equation}
\end{proposition}
\begin{proof}
Observe that  eq. (\ref{Crowell}) together with eq. (\ref{resolution}) and eq. (\ref{eq:inter}) give us that 
\[
    H_1(X,\mathbb{C}) =(s-1) \mathbb{C}[C] 
    + \mathbb{C} - A \oplus B  + A\cap B. 
\]
Therefore, the representation $\chi_\nu$ appears 
\[
    \lambda +(s-2) - 
    \#\{0\leq i \leq s: n \mid \nu d_i (n,d_i) \} + \langle A\cap B, \chi_\nu \rangle.
\]
\[
= \lambda + 
    \#\{0\leq i \leq s: n \nmid \nu d_i (n,d_i) \} + 
    \langle A\cap B, \chi_\nu \rangle-2.
\]
For the trivial representation $\lambda=1$ there is also contribution from $A\cap B\cong \mathbb{C}$, thus  $M_0=1+ 1 -2 =0$. 

For a nontrivial representation we have $M_\nu  = \#\{0\leq i \leq s: n \nmid \nu  (n,d_i) \} -2$. 
The proof is now complete. 
\end{proof}
\subsection{Comparison with Chevaley-Weil formula}
For the Galois module structure of $H^0(X,\Omega_X)$ in the semisimple case the Chevalley-Weil formula \cite{MR3069638}, see \cite{MR589254}, \cite{Naeff2005}. An equivalent treatment in the language of function fields for the case we study is given in \cite[th.2 ]{vm}, where the following formula is proved:

The irreducible representation $\chi_\nu$ of $C$ on $H^0(X,\Omega_X)$ appears 
\[
-1 + \sum_{i=1}^s 
\left\langle 
\frac{-\nu d_i}{e_i}
\right\rangle
+\lambda
=
-1 + \sum_{i=1}^s 
\left\langle 
\frac{-\nu d_i (n,d_i)}{n}
\right\rangle
+\lambda
\]
times, where $\lambda=1$ if $\nu=0$ and $\lambda=0$ otherwise. 
Transferring the notation of \cite{vm} in our notation we have $g_E=0$, $a_k=\nu$ since $r=1$, and $\Phi_i$ is $d_i/(n,d_i)$.

We can use this computation to compute the $\mathbb{C}[C]$-module structure of $H_1(X,\mathbb{C})$ as follows. 
 The space of holomorphic differentials $\Omega^1(X) = H^0(X, \Omega_X^1)$ on a compact Riemann surface $X$ of genus $g$ (where $\dim_{\mathbb{C}} \Omega^1(X) = g$) is isomorphic as a $\mathbb{C}$-vector space to the $\mathbb{C}$-vector space $H^1(X, \mathbb{C})$.
First Serre duality provides a natural isomorphism:
$$H^1(X, \mathcal{O}_X) \cong \Omega^1(X)^*,$$
where $\Omega^1(X)^*$ is the dual space of $\Omega^1(X)$.
For a compact Riemann surface $X$, the {\em Hodge Principle} and the {\em De Rham Isomorphism} yield the relation:
$$
H_{dR}^1(X, \mathbb{C}) \cong H^0(X, \Omega_X^1) \oplus H^1(X, \mathcal{O}_X).
$$
Since $H_{dR}^1(X, \mathbb{C}) \cong H^1(X, \mathbb{C}) \cong H_1(X, \mathbb{C})^*$ (where $H_1(X, \mathbb{C}) = H_1(X, \mathbb{Z}) \otimes_{\mathbb{Z}} \mathbb{C}$), we have the {\em Hodge Decomposition}:
\begin{equation}
    \label{eq:Hodge}
H^1(X, \mathbb{C}) \cong \Omega^1(X) \oplus \overline{\Omega^1(X)},
\end{equation}
($\overline{\Omega^1(X)}$ is the space of anti-holomorphic 1-forms. which is isomorphic to $H^1(X, \mathcal{O}_X)$). Equation (\ref{eq:Hodge})  is also a decomposition of $\mathbb{C}[C]$-modules, that is 
%
the character of $H^1(X, \mathbb{C})$ is:
$$
\chi_{H^1(X, \mathbb{C})} = \chi_{\Omega^1(X)} + \overline{\chi_{\Omega^1(X)}}. 
$$
Thus, the $\mathbb{C}[C]$-module structure of $H^1(X, \mathbb{C})$ (which is the dual of $H_1(X, \mathbb{C})$) is completely determined.


Let us write for the  $\mathbb{C}[C]$-module structure of the homology is:
$$H_1(X, \mathbb{C}) \cong \bigoplus_{\nu=0}^{n-1} M_{\nu} \cdot \chi_{\nu}$$
where $M_{\nu}$ is the multiplicity of $\chi_{\nu}$ in $H^1(X, \mathbb{C})$.
Due to the Hodge decomposition, the multiplicity $M_{\nu}$ is given by:
$$M_{\nu} = \text{mult}(\chi_{\nu}, \Omega^1(X)) + \text{mult}(\chi_{\nu}, \overline{\Omega^1(X)}) = m_{\nu} + m_{n-\nu},$$
where $m_{\nu}$ is the multiplicity of $\chi_{\nu}$ in $\Omega^1(X)$ (from the Chevalley-Weil type formula).

We will now compute the 
multiplicities $M_{\nu}$.

For the trivial character $\chi_0$ ($\nu=0$) we have:
$$m_0 = -1 + \sum_{i=1}^s \left\langle 0 \right\rangle + 1 = 0 \text{ thus } M_0 = m_0 + m_{n-0} = 0 + 0 = 0.$$

For the non-trivial characters $\chi_{\nu}$ ($\nu \in \{1, \dots, n-1\}$) we use:
$$m_{\nu} = -1 + \sum_{i=1}^s \left\langle \frac{-\nu d_i}{n} \right\rangle$$
$$m_{n-\nu} = -1 + \sum_{i=1}^s \left\langle \frac{-(n-\nu) d_i}{n} \right\rangle = -1 + \sum_{i=1}^s \left\langle \frac{\nu d_i}{n} \right\rangle.$$

The total multiplicity $M_{\nu}$ is:
$$M_{\nu} = m_{\nu} + m_{n-\nu} = \left[ -1 + \sum_{i=1}^s \left\langle \frac{-\nu d_i}{n} \right\rangle \right] + \left[ -1 + \sum_{i=1}^s \left\langle \frac{\nu d_i}{n} \right\rangle \right]$$
$$M_{\nu} = -2 + \sum_{i=1}^s \left( \left\langle \frac{-\nu d_i}{n} \right\rangle + \left\langle \frac{\nu d_i}{n} \right\rangle \right).$$

Since $\langle -x \rangle + \langle x \rangle = 1$ if $x \notin \mathbb{Z}$, and $0$ if $x \in \mathbb{Z}$, we conclude:
$$M_{\nu} = -2 + (\text{Number of } i \text{ such that } n \nmid \nu d_i).$$
This is exactly the formula in eq. (\ref{eq:Hstruct}), notice that $n \mid \nu d_i$ if and only if $n \mid \nu (n,d_i)$.


 \def\cprime{$'$}

\end{document}